\documentclass[11pt]{article}

\usepackage{graphicx,color,amsfonts,amsmath,amssymb,enumerate}
\usepackage{url,fancyhdr,indentfirst}
\usepackage{amsthm,natbib,comment}
\usepackage[margin=1in]{geometry}
\usepackage{rotating}
\usepackage{booktabs}
\usepackage{dsfont,bm,mathrsfs}
\usepackage{tikz-qtree,float}
\usepackage{caption}
\usepackage{subcaption}
\usepackage{setspace}
\usepackage{footmisc,multirow}
\usepackage[compact]{titlesec}
\usepackage[colorlinks=true,citecolor=blue]{hyperref}

\def\d{\, \mathrm{d}}

\newcommand{\var}{\mathrm{Var}}

\newcommand{\VaR}{\mathrm{VaR}}

\newcommand{\CVaR}{\mathrm{CVaR}}

\newcommand{\E}{\mathbb{E}}

\newcommand{\R}{\mathbb{R}}
\newcommand{\N}{\mathbb{N}}
\newcommand{\p}{\mathbb{P}}

\renewcommand{\ge}{\geqslant}
\renewcommand{\le}{\leqslant}

\renewcommand{\leq}{\leqslant}
\renewcommand{\epsilon}{\varepsilon}

\renewcommand{\cdots}{\dots}

\newtheorem{example}{Example}

\newtheorem{definition}{Definition}

\theoremstyle{plain}
\newtheorem{theorem}{Theorem}
\newtheorem{lemma}{Lemma}
\newtheorem{proposition}{Proposition}
\newtheorem{corollary}{Corollary}
\theoremstyle{definition}
\theoremstyle{remark}
\newtheorem{remark}{Remark}
    \allowdisplaybreaks
\usepackage{setspace} 
\usepackage{footmisc,multirow,caption}
\renewcommand{\cite}{\citet}

\usepackage[compact]{titlesec}
\titlespacing{\section}{0pt}{8pt}{4pt}
\titlespacing{\subsection}{0pt}{6pt}{4pt}
\titlespacing{\subsubsection}{0pt}{5pt}{3pt}

\begin{document}

\title{Distributionally Robust Reinsurance under Robust Optimized Certainty Equivalent Risk Measure}

\author{Xinqiao Xie\thanks{School of Management, University of Science and Technology of China, China. \url{xxqyzb@mail.ustc.edu.cn}}   \and  Taizhong Hu\thanks{School of Management, University of Science and Technology of China, China. \url{thu@ustc.edu.cn}}   \and  Tiantian Mao\thanks{School of Management, University of Science and Technology of China, China.  \url{tmao@ustc.edu.cn}} }

\maketitle

\begin{abstract}

In this paper, we introduce a class of preference robust risk measures-\emph{robust optimized certainty equivalents} (ROCE)-which encompasses several widely used measures, including Conditional Value-at-Risk and expectiles, as special cases. Motivated by recent developments in distributionally robust optimal reinsurance (DROR), we investigate DROR problems under the ROCE risk measure and consider two prominent uncertainty sets: the mean-variance uncertainty set and the Wasserstein uncertainty set.
For the mean-variance uncertainty set, we reformulate the infinite-dimensional optimization problem into a finite-dimensional one by showing that it suffices to consider three-point distributions. This leads to a unified and explicit formulation for a broad class of ROCE risk measures and offers a simplified framework that also recovers earlier results for Conditional Value-at-Risk and expectiles.
For the Wasserstein uncertainty set, we also derive a tractable finite-dimensional formulation. The resulting data-driven models enable efficient computation and facilitate a systematic comparison between moment-based and Wasserstein-based uncertainty sets in the optimal deductible design. Numerical experiments are exhibited to illustrate the performance of our reformulated programs. 
\medskip


\noindent \emph{MSC2000 subject classification}: 62P05, 91B30


\end{abstract}

\noindent\rule{\textwidth}{0.5pt}

\section{Introduction}

The optimal reinsurance problem has garnered significant attention since the pioneering works of \cite{B60} and \cite{A63}, and has been widely explored in the literature; see, for example, {\cite{CYYZ20}, \cite{YCZ24}, \cite{CLY24}} and
\cite{CC20} for a review. The classic reinsurance problem aims to identify an optimal reinsurance contract that minimizes the total retained risk under a given risk measure. For a ground-up loss $X$, it can be formulated by
\begin{align*}
   \min_{I\in{ \mathcal {I}}}\, \rho (X-I(X)+\pi (I(X)) ),
\end{align*}
where $\rho$ denotes a risk measure, $I$ is a candidate reinsurance contract, and the set of reinsurance contracts $\mathcal{I}$ is usually chosen as the set of admissible stop-loss reinsurance contracts, defined as
\begin{align*}
    \mathcal{I} = \{ I_d : I_d(x) = (x-d)_+, \, d \in \mathbb{R}_+ \},
\end{align*}
with $x_+ := \max\{x,0\}$ for $x\in\R$. The premium principle is given by
$ 
\pi(I(X)) = (1+\theta)\,\mathbb{E}[I(X)],
$ 
where $\theta \in \mathbb{R}_+$ is the safety loading.  
For a stop-loss contract $I(X) = (X-d)_+$, the parameter $d \in \R_+$ is referred to as the \emph{deductible}.
In classic reinsurance design,  the distribution of $X$, faced by the insurer, is assumed to be precisely known.  
In practice, however, the true loss distribution is not directly observable and is inferred from limited data and modeling assumptions. Consequently, statistical estimates are subject to sampling error, model misspecification, and uncertainty in tail behavior. These considerations call for incorporating model uncertainty into reinsurance design and risk evaluation. A well-established approach is the distributionally robust optimization (DRO) framework (see \cite{DY10}, \cite{MK18}), in which decisions are assessed against an uncertainty set of probability measures consistent with the available statistical information. In this paper, we study the following distributionally robust reinsurance optimization (DROR) problem
\begin{equation}
 \label{00}
    \min_{d\in\R_+} \sup_{F \in\mathcal{P}} \rho^F\left(X\wedge d +(1+\theta)\E^F[\left(X-d\right)_+]\right) 
\end{equation}
where $F$ is the cumulative distribution function (cdf) of $X$, and $\mathcal{P}$ is the set of probability distributions, known as the uncertainty set. Here and throughout the paper, the superscript $F$ in $\rho^F$ 
indicates that the corresponding functional is calculated under the assumption that the cdf of $X$ is $F$.  The optimal value of $d$ to problem \eqref{00} is called the optimal deductible and is denoted by $d^*$. The choice of the risk measure used to assess retained losses is crucial for optimal reinsurance design. Existing studies have considered various risk measures within the DRO framework. For example, \cite{LM21} investigate the problem in which $\rho$ is the Value-at-Risk (VaR) or Conditional Value-at-Risk (CVaR), whereas \cite{XLMZ23} focus on the case where 
$\rho$ is the expectile.
We refer readers to \cite{CLY24}, \cite{CJM25} and \cite{BJ24,BJ25a,BJ25} for alternative DROR settings. 

 The \emph{optimized certainty equivalent} (OCE) is a widely used convex risk measure, originally proposed in decision theory by \cite{BT86}. 
For an increasing convex loss (disutility) function $\ell$, the OCE of a loss $X$ is
\begin{align}\label{eq:OCE}  
\rho(X):=\inf _{t \in \mathbb{R}}\left\{t+\mathbb{E}\big[\ell(X-t)\big]\right\}.
\end{align}
As shown in \cite{RU00}, CVaR arises as a special case of the OCE. Using OCE in risk assessment requires specifying an appropriate loss function $\ell$ to represent the decision maker's preferences. 
In practice, however, the available preference information is often partial:  {it may come from qualitative requirements (e.g., risk aversion, convexity, normalization),
or from limited elicitation data such as questionnaire-based assessments that provide only a few certainty equivalents,
or pairwise comparisons between simple lotteries,}  which makes it difficult to identify the true loss function $\ell$.
A standard ambiguity-averse approach in preference robust optimization is to construct an ambiguity set $\Gamma$ of plausible loss functions and evaluate risk under the worst-case element of $\Gamma$; see, e.g., \cite{MMR06}, \cite{AD15} and \cite{GuoXu24}. 
Motivated by this preference robust principle, we introduce a preference robust extension of OCE, referred to as the \emph{robust optimized certainty equivalent} (ROCE), defined by
$$
\rho(X):=\sup_{\ell\in \Gamma}\, \inf_{t\in\mathbb{R}}\left\{t+\mathbb{E}\big[\ell(X-t)\big]\right\}.
$$ 
We will show later that ROCE includes $\CVaR$ and expectiles as special cases. As a preference robust version of OCE, it inherits standard OCE properties such as monotonicity, convexity, translation invariance, and law invariance, which we later leverage for tractability. We refer readers to \cite{BDT20}, \cite{WX22} and \cite{WMH24} for other related generalizations of OCE.\footnote{Related generalizations pursue different robustness axes. \cite{BDT20} develop distributionally robust OCE risk measure over an uncertainty set of probability distributions (e.g., Wasserstein balls), derive convex dual representations, and show that the resulting problems often reduce to finite-dimensional programs with applications to option pricing. \cite{WX22} propose a preference robust modified OCE (RMOCE) that maximizes a utility functional and guards against utility misspecification via an uncertainty set (e.g., Kantorovich balls), yielding an alternating linear-programming scheme and statistical robustness guarantees. It is worth noting that RMOCE coincides with our ROCE risk measure in certain cases-for example, whenever a minimax theorem holds for RMOCE. \cite{WMH24} introduce a generalized OCE based on variational preferences by replacing the linear expectation with a general law-invariant convex risk measure (and linking to rank-dependent utility).} {By employing ROCE in our DROR problem, we incorporate both preference ambiguity and distributional ambiguity into decision making.}

In the DRO framework, it is important to choose an appropriate distributional uncertainty set based on the available information about the true distribution. 
In this paper, we focus on the mean-variance uncertainty set and the Wasserstein ball, both of which are widely used in DRO due to their tractability and data-driven motivation; see \cite{DY10}, \cite{MK18} and references therein. 
Our main contribution is to derive tractable reformulations of the resulting DROR problems under the two uncertainty sets for general ROCE risk principles. 
In particular, since $\CVaR$ and expectiles are special cases of ROCE, we provide a unified  proof of the main result in \cite{LM21} and \cite{XLMZ23}. Besides, we also derive the closed form solution of the DROR problem based on  the mean-CVaR for the mean-variance uncertainty set, which generalizes the result of \cite{LM21}.  

The rest of the paper is organized as follows. In Section \ref{sect-2}, we introduce the ROCE risk measure and present several important examples, including CVaR and expectiles. The finite-dimensional reformulations of problem \eqref{00} under the mean-variance uncertainty set and the Wasserstein uncertainty set are developed in Sections \ref{main-sect} and \ref{main-wasser}, respectively. 
We also illustrate how these results specialize to classical risk measures and discuss extensions to high-dimensional risk in Sections~\ref{main-sect} and \ref{main-wasser}. {We provide numerical experiments in Section \ref{sect-4}, which highlight the performance of our DROR model under the mean-variance uncertainty set compared with classical parametric models.  We also explore the sensitivity of the DROR model under the Wasserstein uncertainty set to the radius, and present a comparison of the out-of-sample behavior of the two DROR models and sample average approximation (SAA) benchmark.} We defer all proofs to Appendix.

\section{Robust Optimized Certainty Equivalent}
\label{sect-2}

Let $(\Omega,\mathcal{F},\p)$ be  a standard probability space and denote by $L_0=L_0(\Omega,\mathcal{F},\p)$ the set of all random variables on $(\Omega,\mathcal{F},\p)$. 
For a set $A\subseteq \R$, let ${\rm int}\,A$ be the set of interior points of $A$, and ${\rm conv}(A)$ be the convex hull of $A$.  For a function $\ell: \R\to\R$, define  $\partial \ell(\R)={\rm conv} \{\ell'(t),\ t \in \R\}$,  where $\ell' $ denotes the left-derivative of $\ell$.

\begin{definition}
 \label{def-ROCE}
Let $\Gamma$ be a set of increasing convex functions $\ell:\R\to\R$ satisfying 
$1\in {\rm int}\,\partial \ell(\R)$ for every $\ell\in\Gamma$. Define $\rho:\mathcal X\to\R$ with $\mathcal X:=\{X\in L_0:\E\!\left[\ell\left(X-t\right)\right]\in\R,\ 
\forall\,t\in\R,\ \forall\,\ell\in\Gamma\}$ as 
\begin{align}\label{rho_p}
\rho(X)=\sup_{\ell\in \Gamma}\, \inf_{t\in\R} \left\{t+\E\left[\ell\left(X-t\right)\right]\right\},~X\in\mathcal X.
\end{align}
 The functional $\rho$ is called a robust optimized certainty equivalent (ROCE) risk measure.
\end{definition}

When the index set $\Gamma$ is a singleton, the risk measure reduces to the classical \emph{optimized certainty equivalent} (OCE) risk measure \eqref{eq:OCE} introduced by \cite{BT86}. 
In general, we allow $\ell$ to vary over a family $\Gamma$ that represents all loss functions compatible with the available preference information, and define ROCE as the worst-case OCE over $\Gamma$. To ensure well-definedness and exclude trivial cases, we impose (and henceforth assume) that
$1\in \mathrm{int}\,\partial \ell(\mathbb{R})$ for all $\ell\in\Gamma$. Specifically, this condition guarantees that, for each $\ell\in\Gamma$, the inner minimization over $t$ in the definition of ROCE \eqref{rho_p} admits a minimizer (i.e., the infimum is attained); see \cite{WMH24} for details. 
We next study several fundamental properties of ROCE as a risk measure.

\begin{proposition}
    Let $\rho$ be an ROCE risk measure defined on $\mathcal X$. Then
    \begin{itemize}
        \item[(i)]{\rm{(Monotonicity)}} If $X,Y\in\mathcal X$ and $X\le Y$, then $\rho(X)\le\rho(Y)$.
        \item[(ii)]{\rm(Translation-invariance)} For every $X\in\mathcal X$ and $c\in\mathbb{R}$, $\rho(X+c)=\rho(X)+c$.
        \item[(iii)]{\rm(Convexity)} For every $X,Y\in\mathcal X$ and $\lambda\in[0,1]$, $\rho(\lambda X+(1-\lambda)Y)\le \lambda\rho(X)+(1-\lambda)\rho(Y)$.
        \item[(iv)]{\rm(Positive homogeneity) }If every $\ell\in\Gamma$ is positively homogeneous, i.e.
        $\ell(\lambda y)=\lambda \ell(y)$ for all $\lambda\ge 0$, and  $y\in\mathbb{R},$
        then for all $\lambda\ge 0$ and $X\in\mathcal X$, $\rho(\lambda X)=\lambda\,\rho(X).$
    \end{itemize}
\end{proposition}
ROCE encompasses many popular risk measures. In Example \ref{ex1} below, we show that ROCE specializes to 
$\CVaR$ and coherent expectiles as special cases.

\begin{example}\label{ex1}
\begin{itemize}
  \item [{\rm (i)}] {\rm (CVaR)}\ \ The $\CVaR$ at level $\alpha\in [0,1]$ is defined as \begin{align*}   \CVaR^F_{\alpha}(X)=\frac{1}{1-\alpha} \int_\alpha^1\VaR^F_\gamma(X)\d\gamma,~ \alpha\in [0,1);~~\CVaR^F_1(X)=\VaR^F_1(X),\end{align*}  where \begin{align*}	\VaR^F_{\alpha}(X):=\inf\{x\in\R:  F(x) > \alpha\},~\alpha\in [0,1);~~\VaR^F_1(X):=\inf\{x\in \R:  F(x) \ge 1\}.\end{align*} By \cite{RU00}, $\CVaR_\alpha$ admits the representation
	\begin{align}\label{eq:RU02}
    	\CVaR^F_\alpha(X)=\inf_{t \in \R} \left\{t+\frac{1}{1-\alpha}\E^F[(X-t)_+]\right\}.
	\end{align}
 It follows that $\CVaR_\alpha$ is an ROCE risk measure with $\Gamma$ being a singleton $\{\ell\}$, where $\ell (z)= {z_+}/({1-\alpha})$ for $z\in\R$.

\item [{\rm (ii)}] {\rm (Expectile)}\ \ For $\beta\in [1/2,1]$, the expectile at level $\beta$ is defined as \begin{align*} 	e^F_\beta(X)=\arg\min_{x \in \R} \left\{\beta \E^F [(X-x)_+^2]         +(1-\beta)\E^F [(x-X)_+^2]\right\}.\end{align*}
By Proposition 9 of \cite{BKMG14} and \eqref{eq:RU02}, $e_\beta$ admits the ROCE representation
 \begin{align*}
     e^F_\beta(X)=&\sup _{\gamma \in\left[{1}/{\nu}, 1\right]}\left\{(1-\gamma) \CVaR^F_{\frac{\nu\gamma-1}{\nu\gamma-\gamma}}(X)+\gamma \E^F[X]\right\}\\
     =&\sup _{\gamma \in\left[{1}/{\nu}, 1\right]}\inf_{t \in \R} \left\{t+{\gamma(\nu-1)}\E^F[(X-t)_+]+\gamma\E^F[X-t]\right\}\\
     =&\sup _{\gamma \in\left[{1}/{\nu}, 1\right]} \inf_{t \in \R} \left\{t+ \E^F[\ell_\gamma(X-t) ] \right\} ,
 \end{align*}
where $\nu:={\beta}/({1-\beta})$ and  
 $\ell_\gamma (z)=\gamma(\nu-1)z_++\gamma z=\gamma\nu z_+-\gamma(-z)_+.$ 
\end{itemize}
\end{example}

As mentioned above, the ROCE risk measure also contains the OCE risk measure as a special case. We list some popular OCE risk measures in the following example.

\begin{example}
  \label{ex2}
\begin{itemize}
  \item[{\rm (i)}] {\rm (Piecewise linear loss function)}\ \ For $ \eta_1$ and $\eta_2$ with $0 \leq \eta_1<1<\eta_2$, define a  piecewise linear loss function
     $$
            \ell (z)=\eta_2 z_+ - \eta_1 (-z)_+, \quad z\in\R.
     $$
     The corresponding OCE risk measure  defined by \eqref{eq:OCE}, introduced in \cite{PR05}, is formulated as
     $$
          \rho^F (X)=\eta_1\E ^F[X]+(1-\eta_1)\CVaR^F_{\xi}(X),
     $$
     where $\xi=({\eta_2-1})/({\eta_2-\eta_1})$. Moreover, the corresponding OCE risk measure is well-known as mean-CVaR risk measure. In particular, when $\eta_1=0$, $\rho$ reduces to $\CVaR_{\xi}$.

\item[{\rm (ii)}] {\rm (Monotone HARA loss function)}\ \ For $\gamma<0$, the monotone HARA loss function, introduced in \cite{CMMR12}, is defined as
    $$
         \ell (z)=\frac{1-(1-z/ \gamma)_+^{1-\gamma}}{1 / \gamma-1}, \quad z\in\R.
    $$
    Note that the risk measure defined by \eqref{eq:OCE} with the monotone HARA loss function $\ell (z)$ is an OCE risk measure. In particular, when $\gamma=-1$, the corresponding OCE risk measure $\rho$ reduces to the {monotone mean-variance}.

\item[{\rm (iii)}] {\rm (Exponential loss function)}\ \ Given $\lambda>0$, the entropic risk measure, introduced in \cite{FS10}, is defined as
    $$   \rho^F (X)=\frac{1}{\lambda}\log \E^F [e^{\lambda X}].   $$
    One can verify that the entropic risk measure is an OCE risk measure generated by the exponential loss function
    $$     \ell(z)=\frac{1}{\lambda}(e^ {\lambda z}-1),~~z\in\R.    $$
\end{itemize}
\end{example}
\begin{remark}
    \cite{WX22} proposed a preference robust modified OCE (RMOCE). Specifically,  given a baseline utility function $v$ for the current consumption level $t$, let $\mathcal U$ be the set of nondecreasing concave utility functions. They consider the RMOCE defined as 
$$
\rho^F(X)=\;\sup_{t\in\mathbb{R}}\,\inf_{u\in\mathcal{U}}\Big\{v(t)+\E^F\big[u(X-t)\big]\Big\}.
$$
We show that RMOCE can be viewed as a generalization of the ROCE risk measure in decision theory; in particular, when a minimax theorem holds, the two formulations coincide.  Moreover, the classical OCE is nested in both ROCE and RMOCE as a special case. Consequently, CVaR and the OCEs in Example~\ref{ex2} are also covered by RMOCE. In addition, expectiles fall under RMOCE, {since the relevant objective function is linear in the preference index $\gamma$}, allowing an application of a minimax theorem. 
\end{remark}

\section{Mean-variance uncertainty set}
\label{main-sect}
The mean-variance uncertainty set goes back to the seminal work of \cite{S58a}, who studied the classical newsvendor problem under moment information (known mean and variance), a formulation that is often viewed as the prototype of DRO. For $\mu,\,\sigma\in\R_+$, we define the mean-variance uncertainty set as
$$
    S(\mu, \sigma) =\left\{F \in \mathcal{P}\left(\mathbb{R}_+\right): \E^{F}[X]=\mu, ~\var^F(X)= \sigma^2\right\},
$$
where $\mathcal{P}\left(\mathbb{R}_+\right)$ denotes the set of all probability distributions supported on $\R_+$. In recent years, such mean-variance uncertainty sets have been widely used in the DRO framework; see, for example, \cite{DY10} and the references therein.

In this section, we focus on the distributionally robust reinsurance optimization problem with the mean-variance uncertainty set $S(\mu,\sigma)$:
\begin{align}\label{eq:main}
	\min_{d\in\R_+} \sup_{F \in S(\mu, \sigma)} \rho^F\left(X\wedge d +(1+\theta)\E^F[\left(X-d\right)_+]\right),
\end{align}
where $\rho: \mathcal{X}\to \R$ is an ROCE risk measure and $\theta\in\R_+$ is the safety loading. We derive a tractable reformulation for a general ROCE risk measure $\rho$ and then apply it to specific cases, including mean-CVaR and expectiles, to obtain further simplified optimization programs.  In particular, we obtain a closed-form solution for mean-CVaR.
\subsection{Reformulation}\label{main-mean-variance}

A key challenge in solving problem \eqref{eq:main} lies in addressing its inner maximization problem: 
\begin{align}\label{eq:main-inner}
	  \sup_{F \in S(\mu, \sigma)} \rho^F\left(X\wedge d +(1+\theta)\E^F[\left(X-d\right)_+]\right).
\end{align} 
Because the supremum is taken over a set of distributions, \eqref{eq:main-inner} is infinite-dimensional and is therefore computationally intractable in its original form. The main result of this section is to derive a tractable reformulation of this inner maximization problem. \eqref{eq:main-inner} is called a worst case distribution. 
Intuitively, we reformulate this inner problem to a finite-dimensional one by showing that 
it suffices to consider three-point distributions. To state it,  define 
$$
    S_3(\mu, \sigma)=\left\{F \in S(\mu,\sigma): F \mbox{ is a three-point cdf}\right\},
$$ 
where a distribution is called a three-point cdf if its support consists of no more than three points.  

\begin{theorem}
 \label{thm-rho_p}
Let $ \rho: \mathcal{X}\to \R$ be an ROCE risk measure with a set of increasing convex functions $\Gamma$, and fix $\mu,\,\sigma\in\R_+$ and $\theta\in\R_+$. Then, for $d\ge 0$, the problem \eqref{eq:main-inner} is equivalent to 
\begin{align}\label{eq:rho_p}
    \sup_{\ell \in \Gamma}~ \inf_{t \in \R}~ \sup_{F \in S_3(\mu, \sigma)} ~\E^F[\tilde{\ell}(X,d,t)],
\end{align}
where  
\begin{align}\label{c_ga_func}
     \tilde{\ell}(x,d,t) := t+\ell\left(x\wedge d-t\right) +(1+\theta) (x-d)_+,~~~ \ell\in\Gamma.
 \end{align} 
\end{theorem}

We next give an observation about problem \eqref{eq:main-inner}: the feasible set $ S(\mu, \sigma)$ can be replaced by $ \overline{S}(\mu, \sigma)$ where
$$
\overline{S}(\mu, \sigma):=\left\{F \in\mathcal P(\R_+):\E^{F}[X]=\mu, ~{\rm Var}^F(X)\le \sigma^2\right\}.
$$ 
\begin{proposition}\label{pro1}
	Fix $d\ge 0$ and $\mu,\,\sigma\in\R_+$, and let $\rho: \mathcal{X}\to \R$ be an ROCE risk measure. Then problem \eqref{eq:main-inner} is equivalent to 
	$$ 
    \sup_{F \in\overline{S}(\mu, \sigma)} \rho^F\left(X \wedge d+(1+\theta)\E^F[(X-d)_+]\right).$$
\end{proposition}

\begin{remark}
	It is worth noting that Proposition \ref{pro1} holds for any  risk measure $\rho$ that is consistent with the increasing convex order; that is,  $\rho(X)\le \rho(Y)$ whenever $X\preceq_{\rm icx} Y$\footnote{For random variables  $X$ and $Y$, we say that $X$ is smaller than $Y$ in increasing convex order, denoted by $X\preceq_{\rm icx} Y$,  if $\E[f(X)] \le \E[f(Y)]$ for any increasing convex function $f$.}.
\end{remark}

Denote by $F=[x_1,p_1;x_2,p_2;x_3,p_3]$ a three-point distribution, where $x_1<x_2< x_3$, $p_i\ge 0$ for $i=1,2,3$, and $p_1+p_2+p_3=1$; that is, if $X\sim F$, then $\p(X=x_i)=p_i$ for $i=1,2,3$. By Theorem \ref{thm-rho_p} and Proposition \ref{pro1}, we can reformulate problem \eqref{eq:main-inner} as follows.
\begin{proposition}\label{pro: convexset}
    Under the condition of Theorem \ref{thm-rho_p},  problem \eqref{eq:main-inner} is equivalent to 
    \begin{align} \label{eq:OCE-three-convex}
\sup_{\ell\in \Gamma}~\inf_{t\in\R} \left\{ \begin{array}{ll} \max   &\sum_{i=1}^3 p_i \tilde{\ell}\left(\frac{y_i}{p_i},d,t\right),\\
 \text {\rm s.t. }  
 &
\sum_{i=1}^3 y_i=\mu, \ \ \sum_{i=1}^3 p_i=1,\ \\
&{\bf y}\in\R_+^3,~ {\bf p}\in\R_+^3 ,~\sum_{i=1}^3 p_i\left(\frac{y_i}{p_i}\right)^2\le\mu^2+\sigma^2
\end{array} \right\},
\end{align}
where $\tilde{\ell}(x,d,t)$ is defined in \eqref{c_ga_func}. We use the convention that $y_i^2/p_i=0$ whenever $y_i=p_i=0$.
\end{proposition}
It is worth noting that the feasible set of the inner maximization problem of \eqref{eq:OCE-three-convex} is a  convex set as $({\bf p},{\bf y})\mapsto \sum_{i=1}^3 p_i\left({y_i}/{p_i}\right)^2$ is a convex function.  
In particular, if $\sup \partial\ell(\mathbb{R})\le 1+\theta$ for all $\ell\in\Gamma$, then the optimal deductible in problem \eqref{eq:main} is $d^*=\infty$.
    More generally, the above conclusion holds for any convex set of distributions $\mathcal{P}$.
 Moreover, if $\Gamma$ is a singleton, that is, $\Gamma=\{\ell\}$, problem \eqref{eq:main} can be written as a minimax problem. 
\begin{corollary}\label{prop:2}
	Under the condition of Proposition \ref{pro: convexset}, if  $\Gamma=\{\ell\}$, then  problem \eqref{eq:main} is equivalent to 
	\begin{align} \label{eq:OCE-single-condition}
 \inf_{d\in\R_+,\,t\in\R} \left\{ \begin{array}{ll} \max   &\sum_{i=1}^3 p_i \tilde{\ell}\left(\frac{y_i}{p_i},d,t\right),\\
 \text {\rm s.t. }  
 &
\sum_{i=1}^3 y_i=\mu, \ \ \sum_{i=1}^3 p_i=1,\ \\
&{\bf y}\in\R_+^3,~ {\bf p}\in\R_+^3 ,~\sum_{i=1}^3 p_i\left(\frac{y_i}{p_i}\right)^2\le\mu^2+\sigma^2
\end{array} \right\},
\end{align}
	where $\tilde{\ell}(x,d,t)$ is defined in \eqref{c_ga_func}.
\end{corollary}

\subsection{Examples}\label{main:sec_sample}
 Applying Theorem \ref{thm-rho_p} to the DROR problem under the risk measures in Examples \ref{ex1} and \ref{ex2}, we have the following results. We first give the closed-form solution to the DROR problem under mean-CVaR.
\begin{proposition}\label{exp+cvar}
	Take $\rho (X)=\eta_1\E[X]+(1-\eta_1)\CVaR_{\xi}(X)$, with $0 \leq \eta_1<1<\eta_2$ and $\xi=({\eta_2-1})/({\eta_2-\eta_1})$.  The following statements hold. 
	\begin{itemize}
		\item [(a)] If $\eta_2<(1+\theta)$,  then the optimal deductible  of problem  \eqref{eq:main} is $d^*=\infty$. 
		\item [(b)] If $\eta_2\ge (1+\theta)$, then the optimal deductible of problem  \eqref{eq:main} is $d^*=0$ if $\theta^*\le \sigma^2/\mu^2$, and $d^*=\mu-\sigma(1-\theta^*)/(2\sqrt{\theta^*})$ otherwise, where $\theta^*=\theta/(1-\eta_1)$.
	\end{itemize}
\end{proposition}
 As mentioned in Example \ref{ex2}, when $\eta_1=0$, we have the risk measure $\rho (X)=\eta_1\E[X]+(1-\eta_1)\CVaR_{\xi}(X)$ reduces to $\CVaR_{\xi}(X)$. Thus, Proposition \ref{exp+cvar} yields a closed-form solution to the DROR problem under $\CVaR$, which recovers the CVaR result in \cite{LM21}. The result is presented in the following corollary. 
\begin{corollary}
 \label{co1}
Take $\rho=\CVaR_\alpha$, $\alpha\in [0,1]$. The following statements hold. 
\begin{itemize}
   \item [{\rm (i)}] If $({1-\alpha})(1+\theta)>1$,  then the optimal deductible  of problem  \eqref{eq:main} is $d^*=\infty$.
   \item [{\rm (ii)}] If $({1-\alpha})(1+\theta)\le 1$, then the optimal deductible of problem  \eqref{eq:main} is $d^*=0$ if $\theta\le \sigma^2/\mu^2$, and $d^*=\mu-\sigma(1-\theta)/(2\sqrt{\theta})$ otherwise.
\end{itemize}  
\end{corollary}

The following result concerns the DROR problem under expectile, which was first studied by \cite{XLMZ23}. Choosing $\rho=e_\beta$ for $\beta \in [1/2,1]$ in  problem \eqref{eq:main}, we obtain the tractable reformulation as stated in the following proposition. 
\begin{proposition}\label{pro_e}
    For $\rho=e_\beta$, $\beta \in [1/2,1]$, and $\nu=\beta/(1-\beta)$, problem \eqref{eq:main} can be reformulated as
\begin{align}\label{eq:ex-temp}
    \min_{d\ge 0} \max_{\gamma \in\left[ {1}/{\nu}, 1\right] }\, \min_{t \in [0,d]}\, g(d,\gamma,t) 
\end{align}  
where if $\gamma>1+\theta$, 
\begin{align*}   g(d,\gamma,t) = \max_{p\in [\sigma^2/(\mu^2+\sigma^2),1]}  p \tilde{\ell}_\gamma \left(\mu-\sigma\sqrt{\frac{1-p}{p}},\,d,\,t\right) + (1- p ) \tilde{\ell}_\gamma \left(\mu+\sigma\sqrt{\frac{p}{1-p}},\,d,\,t\right);\end{align*}if $\gamma\le 1+\theta$, 
\begin{align*}
        g(d, \gamma, t)=\left\{\begin{array}{lll}&\max_{y_i\in\R_+, \, p_i\in\R_+}   \sum_{i=1}^3 p_i \tilde{\ell}_\gamma(y_i, d, t) , \\
     &~~~\text {\rm s.t. } 
\sum_{i=1}^3 y_ip_i=\mu,  \ \ \sum_{i=1}^3 p_i=1, \ \ \sum_{i=1}^3 y_i^2p_i\le\mu^2+\sigma^2,\end{array}\right.
    \end{align*}
with $\tilde{\ell}_\gamma(x,d,t):=t+ \gamma\nu(x\wedge d-t)_+-\gamma (t-x\wedge d)_+ +(1+\theta)(x-d)_+.$ 
In particular, if $\nu\le1+\theta$, that is $\beta \leq \frac{1+\theta}{2+\theta}$, then the optimal deductible  of problem  \eqref{eq:main} is $d^*=\infty$.
\end{proposition}
  In Proposition \ref{pro_e}, we reduce problem \eqref{eq:main} with $\rho=e_\beta$, $\beta\in [1/2,1]$, to a finite-dimensional problem involving only four variables. It is worth noting that this reduced problem differs from the formulation obtained in \cite{XLMZ23} and it is nontrivial to show the equivalence between those two problems.
Applying $\rho(X)=\inf_{x\in\R} \{x+\E [\ell\left(X-x\right) ]\}$, where $\ell (z)= \gamma [1-(1-z/\gamma)_+^{1-\gamma}]/(1-\gamma)$, and $\rho (X)= \lambda^{-1}\log \E [e^{\lambda X}]$, $\lambda>0$, in problem  \eqref{eq:OCE-single-condition}, respectively, we obtain the following corollary.
\begin{corollary}\label{co3}
\begin{itemize}
 \item[{\rm (i)}] For $ \rho(X)=\inf_{t\in\R} \{ t + \E [\ell\left(X-t\right) ]\}$, with $\ell (z) = \gamma [1-(1-z /\gamma)_+^{1-\gamma}]/(1-\gamma)$ with $\gamma<0$, problem \eqref{eq:main} is equivalent to  \begin{align*}
            \min_{0\le t\le d} ~~&  t+ g(d,t),
      \end{align*}
      where
      \begin{align*}
     g(d,t)= \left\{\begin{array}{l}\max_{y_i\in\R_+,\,p_i\in\R_+}~\sum_{i=1}^3p_i\left[ \frac{1-(1-((y_i/p_i)\wedge d-t) / \gamma)_+^{1-\gamma}}{1 / \gamma-1}+(1+\theta)(y_i/p_i-d)_+\right] \\
          ~~~\text {\rm s.t. } \sum_{i=1}^3 y_i=\mu,~~\sum_{i=1}^3 p_i=1, ~~\sum_{i=1}^3 p_i\left(\frac{y_i}{p_i}\right)^2\le \mu^2+\sigma^2. \end{array}\right.
      \end{align*}

 \item[{\rm (ii)}] For $ \rho (X)=\lambda^{-1}\log \E [e^{\lambda X}]$, with $\lambda>0$,  problem \eqref{eq:main} is equivalent to
      \begin{align*}
           \min_{0\le t\le d} ~~&  t+ g(d,t),
      \end{align*}
      where
     \begin{align*}
         g(d,t)= \left\{\begin{array}{l} \max_{y_i\in\R_+,\,p_i\in\R_+}\sum_{i=1}^3 p_i\left[ \frac {1}{\lambda} \left (e^{\lambda ((y_i/p_i)\wedge d-t)}-1\right ) + (1+\theta)((y_i/p_i)-d)_+\right] \\
            ~~~\text {\rm s.t. }\sum_{i=1}^3 y_i=\mu, \ \sum_{i=1}^3 p_i=1, \ \sum_{i=1}^3 p_i\left(\frac{y_i}{p_i}\right)^2\le \mu^2+\sigma^2.\end{array} \right.
     \end{align*}
\end{itemize}
\end{corollary}
\subsection{Extension to High Dimensions}
In practice, a reinsurer typically faces multiple risks, for example through a  portfolio of business lines. The underlying loss is then inherently multivariate $\boldsymbol{X} \in \mathbb{R}_{+}^m$. It is therefore natural to consider a stop-loss contract written on the aggregate loss $Z:=\boldsymbol{\alpha}^{\top} \boldsymbol{X}$, where $\boldsymbol{\alpha} \in \mathbb{R}_{+}^m$ denotes the prespecified weights, and to evaluate the retained loss $Z \wedge d+(1+\theta) \mathbb{E}\left[(Z-d)_{+}\right]$ with an ROCE risk measure in the DROR problem. 
Specifically, for multivariate loss $\boldsymbol{X} \in \mathbb{R}_{+}^m$, we consider the following high-dimensional DROR problem:
\begin{align}\label{eq:high}
	\min_{d\in\R_+} \sup_{F \in \mathcal S(\boldsymbol{\mu}, \Sigma)} \rho^F\left(\left(\boldsymbol{\alpha}^\top\boldsymbol{X}\right)\wedge d +(1+\theta)\E^F\left[\left(\boldsymbol{\alpha}^\top \boldsymbol{X}-d\right)_+\right]\right),
\end{align}
where $\rho:\mathcal{X}\to \R$ is an ROCE risk measure, $\boldsymbol{\alpha} \in \mathbb{R}_{+}^m$, and $F$ is the distribution of $\boldsymbol{X}$. The ambiguity set $\mathcal S(\boldsymbol{\mu},\Sigma)$ is defined by
$$
\mathcal S(\boldsymbol{\mu}, \Sigma)
:=\left\{F \in\mathcal{P}(\mathbb{R}_+^m):\ \mathbb{E}^{F}[\boldsymbol{X}]=\boldsymbol{\mu},\ \mathrm{Cov}^{F}(\boldsymbol{X})=\Sigma\right\},
$$
where $\boldsymbol{\mu}\in\mathbb{R}_+^m$, $\Sigma\succeq 0$ is positive semidefinite, and $\mathcal{P}(\mathbb{R}_+^m)$ denotes the set of all probability distributions supported on $\mathbb{R}_+^m$. By Lemma 2.4 of \cite{CHZ11}, we can show that the inner problem of \eqref{eq:high} reduces to 
\begin{align*}
	\sup_{F \in S(\boldsymbol{\alpha} ^\top\boldsymbol{\mu},~ \sqrt{\boldsymbol{\alpha}^\top\Sigma\boldsymbol{\alpha}})} \rho^F\left(Z\wedge d +(1+\theta)\E^F[\left(Z-d\right)_+]\right),
\end{align*}
where $Z=\boldsymbol{\alpha}^\top \boldsymbol{X}$. Thus, applying Theorem \ref{thm-rho_p}, we obtain that the inner problem of \eqref{eq:high} admits a finite-dimensional reformulation. The result is presented in the following corollary.
\begin{corollary}
    Let $ \rho:\mathcal{X}\to\R$ be an ROCE risk measure with a set of increasing convex functions $\Gamma$, and fix $\boldsymbol{\mu},~\boldsymbol{\alpha}\in\R^m_+$, $\Sigma\succeq 0$ and $\theta\in\R_+$. Then, for $d\ge 0$, we have
\begin{align*}
     &\sup_{F \in \mathcal S(\boldsymbol{\mu}, \Sigma)} \rho^F\left(\left(\boldsymbol{\alpha}^\top \boldsymbol{X}\right) \wedge d+(1+\theta)\E^F\left[\left(\boldsymbol{\alpha}^\top \boldsymbol{X}-d\right)_+\right]\right) \\&~~~=  \sup_{\ell \in \Gamma} \inf_{t \in \R}\, \sup_{F \in S_3(\boldsymbol{\alpha}^\top\boldsymbol{\mu}, \sqrt{\boldsymbol{\alpha}^\top \Sigma\boldsymbol{\alpha}})} \E^F\left[\tilde{\ell}\left(Z,d,t\right)\right],
\end{align*}
where $Z=\boldsymbol{\alpha}^\top \boldsymbol{X}$ and $\tilde{\ell}(x,d,t)$ is defined in \eqref{c_ga_func}.
\end{corollary}

\section{Wasserstein uncertainty set}\label{main-wasser}
The Wasserstein uncertainty set is built on the optimal transport (Wasserstein distance), which goes back to the foundational work of  \cite{K58} and has since become central in optimal transport; see, for example, \cite{V03}, \cite{V08}. In data-driven DRO, Wasserstein balls centered on a nominal (often empirical) distribution provide a flexible way to model distributional uncertainty while preserving tractability; see, for example, \cite{GK23}, \cite{MK18} and the references therein. We begin by recalling the definition of Wasserstein distance. For $m\in\N$ and  $\Xi \subseteq \mathbb{R}^m$, let $\mathcal{P}\left(\Xi\right)$ be the set of all distributions supported within $\Xi$.
\begin{definition}[Wasserstein distance]
  For $p\in[1,\infty)$ and $F_1,F_2\in\mathcal P(\Xi)$, the $p$-Wasserstein distance between  $F_1$ and $F_2$  is defined as
$$
{W}_p(F_1, F_2) = \min_{\pi \in \Pi(F_1, F_2)} \left(\mathbb{E}^{\pi}[\|\xi_1-\xi_2\|^p]\right)^{{1}/{p}},
$$
where  $\|\cdot\|$ is a  norm on $\mathbb{R}^m$ and $\Pi(F_1, F_2)$ is the set of all joint distributions  on $\Xi \times \Xi$ with marginal distributions $F_1$ and $F_2$.
\end{definition}
For the reinsurance problem, we take $\Xi=\mathbb{R}_+$. For  $F_0\in \mathcal P (\mathbb{R}_+)$ and $\varepsilon>0$, the Wasserstein ball centered on $F_0$ with radius $\epsilon$ is defined as
$$
\mathcal{B}_{\varepsilon}(F_0)=\left\{F \in \mathcal{P}\left(\mathbb{R}_+\right) : W_p(F, F_0) \leqslant \varepsilon\right\}.
$$

In this section, 
we focus on the distributionally robust reinsurance optimization problem with Wasserstein ball $\mathcal{B}_{\varepsilon}(F_0)$:
\begin{align}\label{eq:main_wasser}
	\min_{d\in\R_+} \sup_{F \in \mathcal{B}_{\varepsilon}(F_0)} \rho^F\left(X\wedge d +(1+\theta)\E^F[\left(X-d\right)_+]\right),
\end{align}
where $\rho:\mathcal{X}\to \R$ is an ROCE risk measure. We first give an equivalent reformulation of the inner problem of \eqref{eq:main_wasser} for a general center distribution and a general convex loss function. We then focus on the empirical center distribution and the piecewise-linear convex loss function to derive a tractable reformulation, and apply it to mean-CVaR and expectiles, to obtain further simplified optimization programs.

\subsection{Reformulation}\label{main-wasser-1}
In this section, we aim to reformulate the above problem to a finite-dimensional problem for a class of loss functions under the Wasserstein ball centered on a discrete distribution. A key challenge in solving problem \eqref{eq:main_wasser} lies in addressing its inner maximization problem: 
\begin{align}\label{eq:main-wasser-inner}
	  \sup_{F \in \mathcal{B}_{\varepsilon}(F_0)} \rho^F\left(X\wedge d +(1+\theta)\E^F[\left(X-d\right)_+]\right).
\end{align} 
  The following proposition provides the equivalent reformulation of problem \eqref{eq:main-wasser-inner}.
\begin{proposition}\label{switch_wasser}
   Let $ \rho:\mathcal{X}\to \R$ be an ROCE risk measure with a set of increasing convex functions $\Gamma$, and fix $\varepsilon\in\R_+$, $\theta\in\R_+$ and a center distribution $F_0$. For $d\ge 0$, we have that problem \eqref{eq:main-wasser-inner} is equivalent to
\begin{align}\label{eq:switch_wasser}
     \sup_{\ell \in \Gamma} \inf_{t \in \R}\, \sup_{F \in \mathcal{B}_{\varepsilon}(F_0)} \E^F\left[\tilde{\ell}\left(X,d,t\right)\right],
\end{align}
where $\tilde{\ell}(x,d,t)$ is defined in \eqref{c_ga_func}.
\end{proposition}
 Problem \eqref{eq:main-wasser-inner} remains infinite-dimensional and is therefore computationally intractable both in its original formulation and in the representation \eqref{eq:switch_wasser}. Based on the above equivalent reformulation, we next focus on the specific reference distribution and specific loss function to derive a tractable reformulation of problem \eqref{eq:main-wasser-inner}. To leverage observed data in reinsurance decision-making, we consider the data-driven distributionally robust optimization framework under a Wasserstein ball centered at the empirical distribution. Specifically, given sample points $\{\widehat{x}_i\}_{i=1}^N$, we take the reference distribution as $F_0=\frac{1}{N}\sum_{i=1}^N\delta_{\widehat{x}_i}$, where $\delta_{x}$ denotes the Dirac distribution placing unit mass at $x$. Moreover, since any convex function can be approximated arbitrarily well by piecewise-linear convex functions, and such functions are computationally convenient, we consider using piecewise-linear convex loss function in DROR problem under Wasserstein ball. Specifically, for $\ell\in\Gamma$, we assume that $\ell(y)=\max_{k\le K}a_{\ell,k}y+b_{\ell,k}$, where $0\le a_{\ell,1}\le\ldots\le a_{\ell,K}$. We also  assume that there exist $-\infty=h_0<h_1<\ldots<h_{K}<h_{K+1}=\infty$ such that  $\ell(y)=a_{\ell,k}y+b_{\ell,k}$ for $y\in (h_{k-1},h_k]$. Based on these assumptions, applying the worst-case expectation result in \cite{KENS19}, we find that the inner problem of \eqref{eq:main_wasser} can be reduced to a finite-dimensional problem. This main result is presented in the following theorem. 
\begin{theorem}\label{pro_wasser_ball_general}
Fix $\varepsilon,\theta\in\R_+$ and samples $\{\widehat x_i\}_{i=1}^N$, and let 
$F_0:=\frac1N\sum_{i=1}^N\delta_{\widehat x_i}$ be the empirical distribution.
Let $\rho:\mathcal X\to\R$ be an ROCE risk measure with loss family $\Gamma$.
Assume that each $\ell\in\Gamma$ is increasing piecewise linear with $K$ pieces, i.e., $\ell(y)=\max_{1\le k\le K}\{a_{\ell,k}y+b_{\ell,k}\}$ with 
$0\le a_{\ell,1}\le\cdots\le a_{\ell,K}$; equivalently, there exist breakpoints $-\infty=h_0<h_1<\cdots<h_{K}<h_{K+1}=\infty$
such that $\ell(y)=a_{\ell,k}y+b_{\ell,k}$ for $y\in(h_{k-1},h_k]$.
Let $p\in[1,\infty]$ and $q$ be its conjugate exponent ($1/p+1/q=1$), and define  $\phi(q):=(q-1)^{q-1}/q^q$ for $q>1$ and $\phi(1):=1$. Then, for any $d\ge 0$, we have
\begin{align}\label{eq:wasser_ball_general}
     &\sup_{F \in \mathcal{B}_{\varepsilon}(F_0)} \rho^F\left(X \wedge d+(1+\theta)\E^F\left[\left(X-d\right)_+\right] \right)    =\sup_{\ell\in\Gamma}\inf _{t\in\mathbb{R}}g_\ell(d,t),
\end{align}
where $g_\ell(d,t)$ is defined as follows.  \begin{itemize}
\item[(i)] If $t\in[d-h_k,d-h_{k-1})$ and $a_{\ell,k}\le 1+\theta$,
\begin{align*}  
g_\ell(d,t)=
\left\{\begin{array}{lll}
&\inf~~   \lambda \varepsilon^p+\frac{1}{N} \sum_{i=1}^N s_i&  \\
&\text {\rm s.t. }~\lambda\in\mathbb{R}_+, s_i\in\mathbb{R}, z_{i,j}\in\R &\forall i \leq N,~\forall j\le k+1,\\
&~~~~~~~(1-a_{\ell,j})t-\widehat{x}_iz_{i, j}+b_{\ell,j}+\phi(q)\lambda\left|\frac{z_{i,j}}{\lambda}\right|^q\leq s_i &\forall i \leq N,~\forall j\le k,\\
&~~~~~~~~z_{i,j}\leq -a_{\ell,j} &\forall i \leq N,~ \forall j\le k,\\
&~~~~~~~ (1-a_{\ell,k})t-\widehat{x}_iz_{i,k+1}\\
&~~~~~~~+b_{\ell,k}+(a_{\ell,k}-(1+\theta))d+\phi(q)\lambda\left|\frac{z_{i,k+1}}{\lambda}\right|^q \leq s_i &\forall i \leq N,\\
&~~~~~~~~z_{i,k+1} \leq -(1+\theta) &\forall i \leq N.
\end{array}\right.
\end{align*}
\item[(ii)] If $t\in[d-h_k,d-h_{k-1})$ and $a_{\ell,k}> 1+\theta$,
\begin{align*}
g_\ell(d,t)=\left\{\begin{array}{lll}&\inf~ \lambda \varepsilon^p+\frac{1}{N} \sum_{i=1}^N s_i & \\
&\text {\rm s.t. }\lambda\in\mathbb{R}_+,v_{i, j}\in\mathbb{R}_-, s_i\in\R, z_{i,j}\in\mathbb{R}, &\forall i \leq N,~\forall j\le k,\\  
&~~~~~~(1-a_{\ell,j})t-\widehat{x}_iz_{i, j}\\
&~~~~~+(z_{i,j}-v_{i,j}+a_{\ell,j})d+b_{\ell,j}+\phi(q)\lambda\left|\frac{z_{i,j}}{\lambda}\right|^q\leq s_i &\forall i \leq N,~\forall j\le k,\\
&~~~~~-a_{\ell,j}\le z_{i,j}-v_{i,j} \leq 0 &\forall i \leq N,~ \forall j\le k-1,\\
&~~~~~-a_{\ell,k}\le z_{i,k}-v_{i,k}\le -(1+\theta) &\forall i \leq N.
\end{array}\right.
\end{align*}
\end{itemize}
\end{theorem}

\subsection{Examples}\label{wasser_sample}
Applying Theorem \ref{pro_wasser_ball_general} to the DROR problem under the risk measures in Examples \ref{ex1} and \ref{ex2}, we have the following results.  First, we give a tractable reformulation of the DROR problem under mean-CVaR.
\begin{proposition}\label{wasser_CVaR}
   Take $\rho (X)=\eta_1\E[X]+(1-\eta_1)\CVaR_{\xi}(X)$, with $0 \leq \eta_1<1<\eta_2$ and $\xi=({\eta_2-1})/({\eta_2-\eta_1})$. The following statements hold.  
	\begin{itemize}
		\item [(a)] If $\eta_2\le 1+\theta$,  then the optimal deductible  of problem  \eqref{eq:main_wasser} is $d^*=\infty$. 
		\item [(b)] If $\eta_2> 1+\theta$, then problem \eqref{eq:main_wasser} admits the finite-dimensional reformulation
        \begin{align*}
        \left\{\begin{array}{lll}&\inf~  \lambda \varepsilon^p+\frac{1}{N} \sum_{i=1}^N s_i & \\
&\text {\rm s.t. } d\in\mathbb{R}_+,t\in[0,d],\lambda\in\mathbb{R}_+, v_{i, 2}\in\mathbb{R}_-,s_i\in\R, z_{i, 1},z_{i, 2}\in\mathbb{R}&\forall i\leq N,\\
&~~~~~(1-\eta_1)t-\widehat{x}_iz_{i, 1}+\phi(q)\lambda\left|\frac{z_{i,1}}{\lambda}\right|^q \leq s_i &\forall i \leq N,\\
&~~~~~z_{i,1}\le -\eta_1 &\forall i \leq N,\\
&~~~~~(1-\eta_2)t-\widehat{x}_iz_{i, 2}+(z_{i,2}-v_{i,2}+\eta_2)d+\phi(q)\lambda\left|\frac{z_{i,2}}{\lambda}\right|^q \leq s_i &\forall i \leq N,\\
&~~~~-\eta_2\le z_{i,2}-v_{i,2}\le -(1+\theta) &\forall i \leq N.
\end{array}\right.
        \end{align*}
	\end{itemize} 
\end{proposition}
By Proposition \ref{wasser_CVaR},
we immediately obtain the following corollary for CVaR. 
\begin{corollary}
Take $\rho={\rm CVaR}_\alpha$, $\alpha\in [0,1]$. The following statements hold.  
\begin{itemize}
   \item [{\rm (i)}] If $({1-\alpha})(1+\theta)\ge1$,  then the optimal deductible  of problem  \eqref{eq:main_wasser} is $d^*=\infty$.
   \item [{\rm (ii)}] If $({1-\alpha})(1+\theta)<1$, then problem \eqref{eq:main_wasser} admits the finite-dimensional reformulation
        \begin{align*}
        \left\{\begin{array}{lll}&\inf~  \lambda \varepsilon^p+\frac{1}{N} \sum_{i=1}^N s_i & \\
&\text {\rm s.t. }d\in\mathbb{R}_+,t\in[0,d],\lambda\in\mathbb{R}_+, v_{i, 2}\in\mathbb{R}_-,s_i, z_{i, 1},z_{i, 2}\in\mathbb{R}&\forall i \leq N,\\&~~~~~ t-\widehat{x}_iz_{i, 1}+\phi(q)\lambda\left|\frac{z_{i,1}}{\lambda}\right|^q \leq s_i &\forall i \leq N,\\
&~~~~~z_{i,1}\le 0 &\forall i \leq N,\\
&~~~~-\frac{\alpha}{1-\alpha}t-\widehat{x}_iz_{i, 2}+(z_{i,2}-v_{i,2}+\frac{1}{1-\alpha})d+\phi(q)\lambda\left|\frac{z_{i,2}}{\lambda}\right|^q \leq s_i &\forall i \leq N,\\
&~~~~-\frac{1}{1-\alpha}\le z_{i,2}-v_{i,2}\le -(1+\theta) &\forall i \leq N.
\end{array}\right.
        \end{align*}
\end{itemize}  
\end{corollary} 

Based on Theorem \ref{pro_wasser_ball_general}, when the risk measure is taken as the coherent expectile, that is, $\rho=e_\beta$  with $\beta\ge 1/2$, we can further reformulate  the DROR problem as follows. 
\begin{proposition}\label{expectile_wasser}
     For $\rho=e_\beta$, $\beta \in [1/2,1]$, and $\nu=\beta/(1-\beta)$, the following statements hold.
     \begin{itemize}
		\item [(a)] If $\nu\le(1+\theta)$,  then the optimal deductible  of problem  \eqref{eq:main_wasser} is $d^*=\infty$. 
		\item [(b)] If $\nu> (1+\theta)$, then the inner problem of  \eqref{eq:main_wasser} admits the finite-dimensional reformulation
        \begin{align}
    \max\left\{\sup_{\gamma\in[1/\nu,(1+\theta)/\nu]}g_1(d,\gamma),\sup_{\gamma\in((1+\theta)/\nu,1]}g_2(d,\gamma)\right\},
\end{align}  
where  $g_1(d,\gamma)$ and $g_2(d,\gamma)$ are   the optimal values of the following two problems 
\begin{align*}  
 \left\{\begin{array}{lll}&\inf~  \lambda \varepsilon^p+\frac{1}{N} \sum_{i=1}^N s_i & \\
&\text {\rm s.t. }t\in[0,d],\lambda\in\mathbb{R}_+, s_i,z_{i, 1},z_{i, 2},z_{i, 3}\in\mathbb{R}&\forall i \leq N,\\
&~~~~(1-\gamma)t-\widehat{x}_iz_{i, 1}+\phi(q)\lambda\left|\frac{z_{i,1}}{\lambda}\right|^q \leq s_i &\forall i \leq N,\\
&~~~~~z_{i,1}\le -\gamma &\forall i \leq N,\\
&~~~~(1-\gamma\nu)t-\widehat{x}_iz_{i, 2}+\phi(q)\lambda\left|\frac{z_{i,2}}{\lambda}\right|^q \leq s_i &\forall i \leq N,\\
&~~~~~z_{i,2}\le -\gamma\nu &\forall i \leq N,\\
&~~~~(1-\gamma\nu)t+(\gamma\nu-(1+\theta))d-\widehat{x}_iz_{i, 3}+\phi(q)\lambda\left|\frac{z_{i,3}}{\lambda}\right|^q \leq s_i &\forall i \leq N,\\
&~~~~~z_{i,3}\le -(1+\theta) &\forall i \leq N,\end{array}\right.\end{align*} 
and 
\begin{align*}\left\{\begin{array}{lll}  
 &\inf~  \lambda \varepsilon^p+\frac{1}{N} \sum_{i=1}^N s_i & \\
&\text {\rm s.t. }t\in[0,d],\lambda\in\mathbb{R}_+,  v_{i, 2}\in\mathbb{R}_-,s_i,z_{i, 1},z_{i, 2}\in\mathbb{R}&\forall i \leq N,\\
&~~~~~(1-\gamma)t-\widehat{x}_iz_{i, 1}+\phi(q)\lambda\left|\frac{z_{i,1}}{\lambda}\right|^q \leq s_i &\forall i \leq N,\\
&~~~~~z_{i,1}\le -\gamma &\forall i \leq N,\\
&~~~~~(1-\gamma\nu)t+(z_{i,2}-v_{i,2}+\gamma\nu)d-\widehat{x}_iz_{i, 2}+\phi(q)\lambda\left|\frac{z_{i,2}}{\lambda}\right|^q \leq s_i &\forall i \leq N,\\
&~~~~~-\gamma\nu\le z_{i,2}-v_{i,2}\le -(1+\theta) &\forall i \leq N,\end{array}\right.\end{align*} 
	\end{itemize} 
    respectively.
\end{proposition}
    Since any convex loss function can be approximated arbitrarily well by a piecewise-linear convex function, Theorem~\ref{pro_wasser_ball_general} yields a natural finite-dimensional approximation of the Wasserstein-robust DROR problem for ROCE risk measures generated by general convex losses. Specifically, we replace the original loss by a piecewise-linear approximation and then apply Theorem~\ref{pro_wasser_ball_general} to obtain a tractable finite-dimensional reformulation. In particular, this construction provides finite-dimensional approximations for the Wasserstein-robust DROR problem under the risk measures in Example~\ref{ex2}, for which the generating loss functions are not piecewise linear.

\subsection{Extension to High Dimensions}
In practice, a reinsurer typically faces multiple risks, for example through a  portfolio of business lines. The underlying loss is then inherently multivariate $\boldsymbol{X} \in \mathbb{R}_{+}^m$. It is therefore natural to consider a stop-loss contract written on the aggregate loss $Z=\boldsymbol{\alpha}^{\top} \boldsymbol{X}$, where $\boldsymbol{\alpha} \in \mathbb{R}_{+}^m$ denotes the prespecified weights, and to evaluate the retained loss $Z \wedge d+(1+\theta) \mathbb{E}\left[(Z-d)_{+}\right]$ with an ROCE risk measure in the DROR problem.
Specifically, for multivariate loss $\boldsymbol{X} \in \mathbb{R}_{+}^m$, we consider the following high-dimensional DROR problem:
\begin{align}\label{eq:high_wasser}
	\min_{d\in\R_+} \sup_{F \in \mathcal{B}_\epsilon(F_0)} \rho^F\left(\left(\boldsymbol{\alpha}^\top\boldsymbol{X}\right)\wedge d +(1+\theta)\E^F\left[\left(\boldsymbol{\alpha}^\top \boldsymbol{X}-d\right)_+\right]\right),
\end{align}
where $\rho:\mathcal{X}\to \R$ is an ROCE risk measure, $F_0$ is a given distribution, $\boldsymbol{\alpha} \in \mathbb{R}_{+}^m$, $F$ is the distribution of $\boldsymbol{X}$, and $\mathcal{B}_\epsilon(F_0)$ is the uncertainty set, defined as
$$
   \mathcal{B}_{\varepsilon}(F_0)=\left\{F \in \mathcal{P}\left(\mathbb{R}_+^m\right) : W_p(F, F_0) \leqslant \varepsilon\right\}.
$$
Applying Proposition 1 of \cite{WLM25}, one can verify that the inner problem of  \eqref{eq:high_wasser} can be written as 
\begin{align*}
	\sup_{F \in \mathcal{B}_{\|\boldsymbol{\alpha}\|_*\varepsilon}(F_0^{\boldsymbol{\alpha}})} \rho^F\left(Z\wedge d +(1+\theta)\E^F[\left(Z-d\right)_+]\right),
\end{align*}
where $Z=\boldsymbol{\alpha}^\top \boldsymbol{X}$, $F_0^{\boldsymbol{\alpha}}$ denotes the distribution of $Z$ induced by $F_0$ (i.e., the pushforward of $F_0$ under $\boldsymbol{x}\mapsto \boldsymbol{\alpha}^\top \boldsymbol{x}$), and $F$ denotes the distribution of $Z$. Here, $\|\cdot\|_*$ is the dual norm of $\|\cdot\|$, defined by $\|\boldsymbol{\alpha}\|_{*}=\sup_{\|\boldsymbol{z}\|\le 1}\boldsymbol{\alpha}^{\top}\boldsymbol{z}$. Thus, by Proposition \ref{switch_wasser}, we obtain the following reformulation. 
\begin{corollary}
    Let $ \rho:\mathcal{X}\to\R$ be an ROCE risk measure generated by a set $\Gamma$ of increasing and convex functions. Given a distribution $F_0$, fix $\epsilon> 0$, $\theta\in\R_+$ and $\boldsymbol{\alpha}\in\mathbb{R}_+^m$. For $d\ge 0$, we have
\begin{align*}
     \sup_{F \in \mathcal{B}_\epsilon(F_0)} \rho^F\left(\left(\boldsymbol{\alpha}^\top \boldsymbol{X}\right)\wedge d+(1+\theta)\E^F\left[\left(\boldsymbol{\alpha}^\top \boldsymbol{X}-d\right)_+\right]\right) =  \sup_{\ell \in \Gamma} \inf_{t \in \R}\, \sup_{F \in \mathcal{B}_{\|\boldsymbol{\alpha}\|_*\epsilon}(F_0^{\boldsymbol{\alpha}})} \E^F\left[\tilde{\ell}\left(Z,d,t\right)\right],
\end{align*}
where $Z=\boldsymbol{\alpha}^\top \boldsymbol{X}$ and $\tilde{\ell}(x,d,t)$ is defined in \eqref{c_ga_func}.
\end{corollary}
Moreover, if $F_0$ is the empirical distribution and the loss function is the piecewise-linear convex function, it follows from Theorem~\ref{pro_wasser_ball_general} that the inner problem in \eqref{eq:high_wasser} admits a tractable reformulation. Specifically, for high-dimensional samples $\{\widehat{\boldsymbol{x}}_i\}_{i=1}^N\subset\mathbb{R}_+^m$, define the projected one-dimensional samples $\widehat{z}_i:=\boldsymbol{\alpha}^\top \widehat{\boldsymbol{x}}_i$. Then the tractable reformulation in Theorem~\ref{pro_wasser_ball_general} applies after replacing $\widehat{x}_i$ by $\widehat{s}_i$ and replacing the Wasserstein radius $\varepsilon$ by $\|\boldsymbol{\alpha}\|_*\varepsilon$. For brevity, we omit the resulting corollary. 

\section{Numerical examples}\label{sect-4} 
In this section, we present numerical experiments to investigate problem \eqref{00} under the mean-variance uncertainty set and Wasserstein ball when  the  risk measure $\rho$ is mean-CVaR: $\rho(X)=\eta_1\E[X]+(1-\eta_1)\CVaR_\xi(X)$, where $0\le\eta_1<1<\eta_2$ and $\xi=(\eta_2-1)/(\eta_2-\eta_1)$. First, we compare our distributionally robust risk measurement under mean-variance uncertainty set with the risk measurement obtained in the classical reinsurance model in which the distribution is assumed to follow a Gamma, Lognormal or Pareto distribution. 
Next, we study the sensitivity of the Wasserstein-robust design to the radius of the Wasserstein ball. Finally, for a fixed dataset, we compare the out-of-sample performance of the two robust models and the sample average approximation (SAA) benchmark. 
For the mean-variance ambiguity set, we benchmark our mean-variance distributionally robust reinsurance model against classical parametric specifications by comparing the resulting risk measure curves and their optimal deductibles. The classical models assume that the loss follows one of the standard actuarial families-Gamma, Lognormal, or Pareto-calibrated to the same mean and variance.
\begin{figure}[ht]
  \centering
  \includegraphics[width=0.8\textwidth]{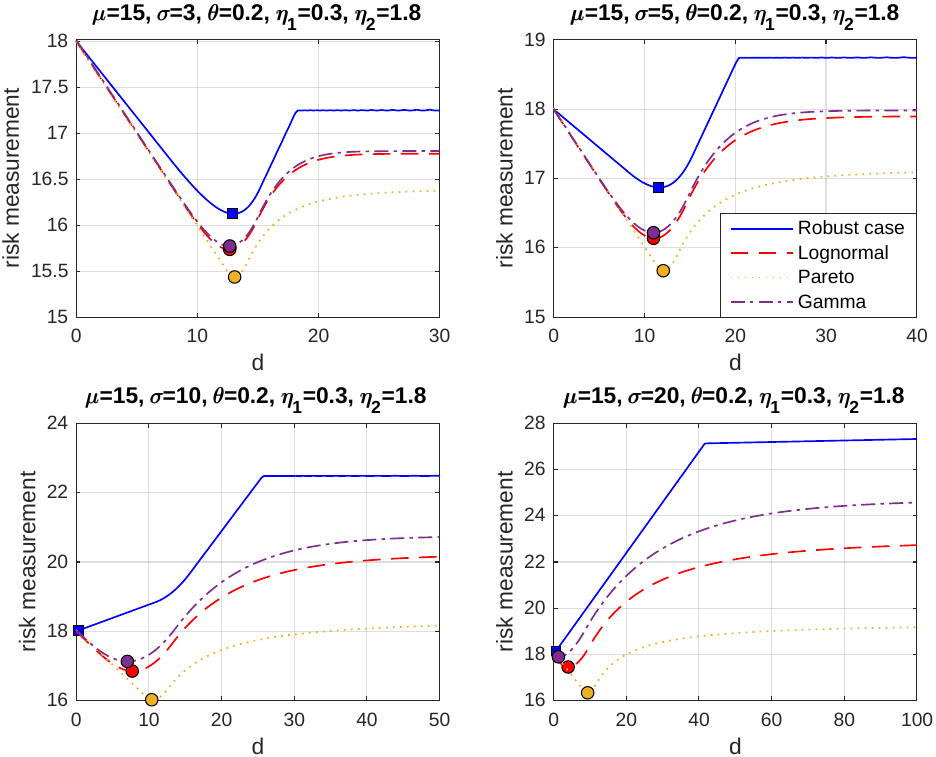}
  \caption{Comparison of risk measurement values in the robust (solid) and non‐robust cases (Lognormal: dashed; Pareto: dotted; Gamma: dash-dot).}
  \label{fig:risk-comparison}
\end{figure}
Figure \ref{fig:risk-comparison} provides the risk measurement across proposed robust model and classical parametric benchmarks. 
Each subplot corresponds to a different level of standard deviation, $\sigma = 3, 5, 10, 20$, with $\mu = 15$, $\theta = 0.2$, and risk parameters $\eta_1 = 0.3$, $\eta_2 = 1.8$ fixed. The robust risk measurement (blue curve) consistently lies above those of the reference distributions, confirming its conservative stance under uncertainty. The optimal deductibles for each model are indicated by their respective markers, with the robust optimum shown as a square. Notably, for smaller standard deviation ($\sigma=3,5$), the optimal deductible of the robust model 
is neither the smallest nor the largest optimal deductible among the parametric benchmarks, and the resulting contract reflects a trade-off between risk transfer and premium loading. 
As $\sigma$ increases, the robust risk curve shifts upward and the minimizer moves toward smaller deductibles. For sufficiently large $\sigma$, the robust solution becomes the smallest optimal deductible relative to the parametric benchmarks. Moreover, for $\sigma=10,\,20$, the robust optimal deductible reaches the boundary $d^*=0$, corresponding to full reinsurance. 

 For the Wasserstein ball, the radius $\varepsilon$ quantifies the size of distributional ambiguity: a larger $\varepsilon$ allows a wider range of models around the center distribution, reflecting less confidence in the center distribution. Thus, we next study the sensitivity of the type-2 Wasserstein distributionally robust reinsurance solution to $\varepsilon$. For fixed safety loading parameter $\theta = 0.2$, and risk parameters $\eta_1 = 0.3$, $\eta_2 = 1.8$, Figure \ref{fig:radius-sensitivity} plots the worst-case risk as a function of the deductible for several Wasserstein radii, under two different empirical center distributions generated from Lognormal and Pareto samples with the same mean and standard deviation ($\mu=15$, $\sigma=5$), respectively.  As $\varepsilon$ increases, the entire objective curve shifts upward and its minimizer shifts toward larger deductibles. Consequently, both the robust optimal value and the optimal deductible $d^*(\varepsilon)$ increase with $\varepsilon$. 
 \begin{figure}[htbp]
  \centering
  \begin{subfigure}[b]{0.48\textwidth}
    \centering
    \includegraphics[width=\textwidth]{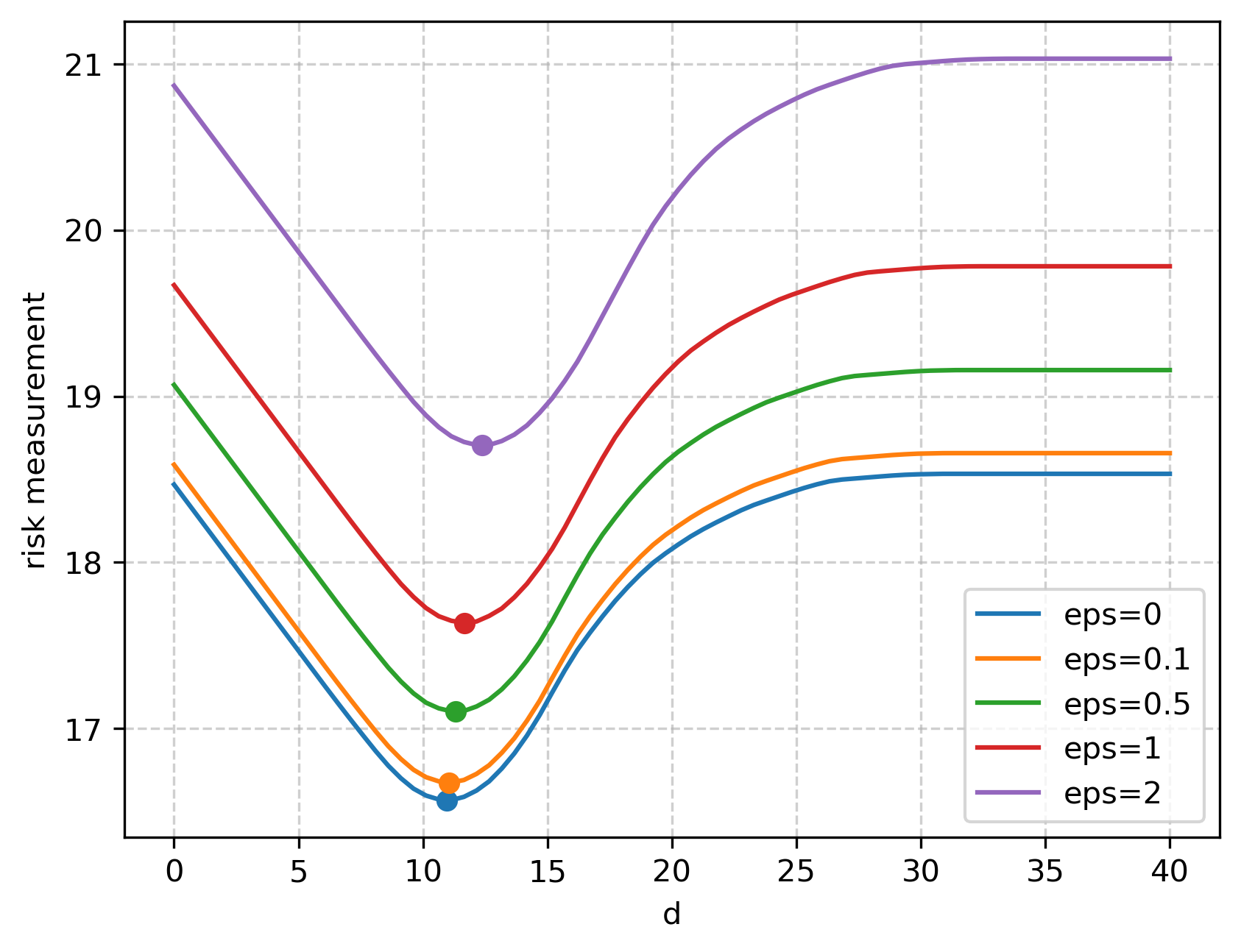}
    \caption{Lognormal center distribution}
    \label{fig:radius-sensitivity-lognormal}
  \end{subfigure}
  \hfill
  \begin{subfigure}[b]{0.48\textwidth}
    \centering
    \includegraphics[width=\textwidth]{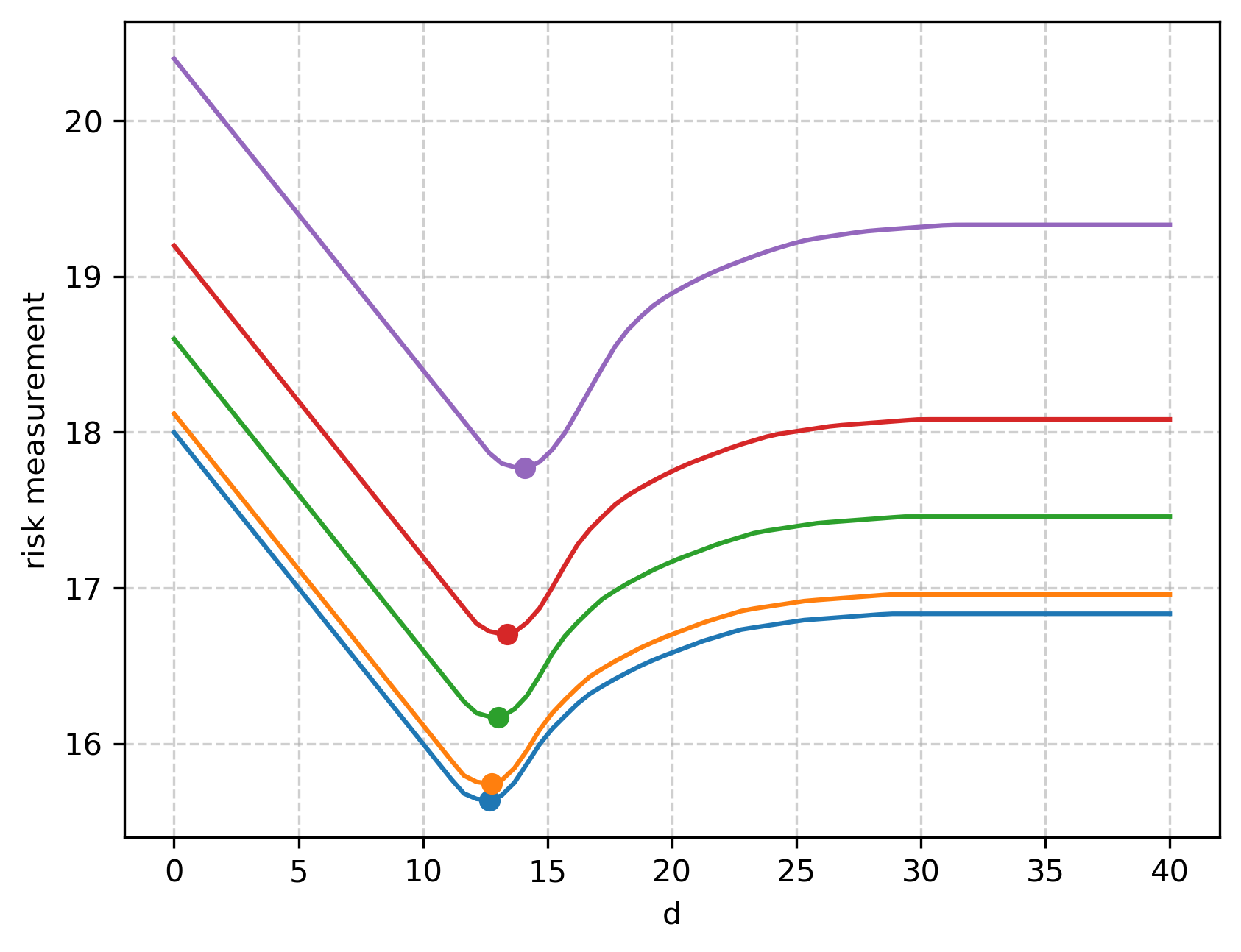}
    \caption{Pareto center distribution}
    \label{fig:radius-sensitivity-pareto}
  \end{subfigure}

  \caption{Worst-case risk as a function of the deductible for different Wasserstein radii, using empirical center distributions generated from a lognormal sample (left) and a Pareto sample (right).}
  \label{fig:radius-sensitivity}
\end{figure}
 
 To end this section, we compare the out-of-sample performance of the mean-variance robust model, the Wasserstein robust model, and the SAA model under model misspecification. A training sample $\left\{\widehat{x}_i\right\}_{i=1}^N$ of size $N=20$ is generated from a Lognormal distribution with prespecified mean and standard deviation $(\mu, \sigma)=(2,0.5)$. From this sample, we construct the empirical distribution, which serves as the center distribution of the type-2 Wasserstein ball, and we use the empirical moments $(\widehat{\mu}, \widehat{\sigma})$ to define the mean-variance uncertainty set. To assess robustness to tail misspecification, the out-of-sample risk $R(d)$ is then evaluated on an independent sample $\left\{\widehat{x}_i'\right\}_{i=1}^N$ drawn from a Pareto distribution with the same mean and standard deviation as the above Lognormal distribution. For each radius $\varepsilon$, let $d_{W_2}^*(\varepsilon)$ denote the optimal deductible obtained from the Wasserstein robust model, and let $d_{\text {MV}}^*$ and $d_{\text {SAA}}^*$ be the optimal deductible from the mean-variance robust model and the SAA model, respectively. Figure \ref{fig:true-risk-misspec} plots the corresponding out-of-sample risks $R\left(d_{W_2}^*(\varepsilon)\right)$ (solid line), $R\left(d_{\text {MV}}^*\right)$ (dashed horizontal line) and $R\left(d_{\text {SAA}}^*\right)$ (dash-dot horizontal line). 
 The mean-variance robust model yields a uniformly smaller out-of-sample risk $R(d^*_{\mathrm{MV}})$ than the  out-of-sample risk $R(d^*_{\mathrm{SAA}})$ of SAA model, indicating a clear improvement over SAA. On the other hand, the Wasserstein-robust performance $R(d^*_{W_2}(\varepsilon))$ depends critically on the radius $\varepsilon$: starting at $\varepsilon=0$ (which coincides with SAA), the out-of-sample risk initially decreases and attains its minimum at an intermediate radius (around $\varepsilon\approx 0.4$), outperforming both mean-variance robust model and SAA. As $\varepsilon$ increases further, $R(d^*_{W_2}(\varepsilon))$ rises steadily, eventually exceeding first the mean-variance benchmark and then the SAA benchmark, and levels off for large radii. 

\begin{figure}[ht]
  \centering
  \includegraphics[width=0.8\textwidth]{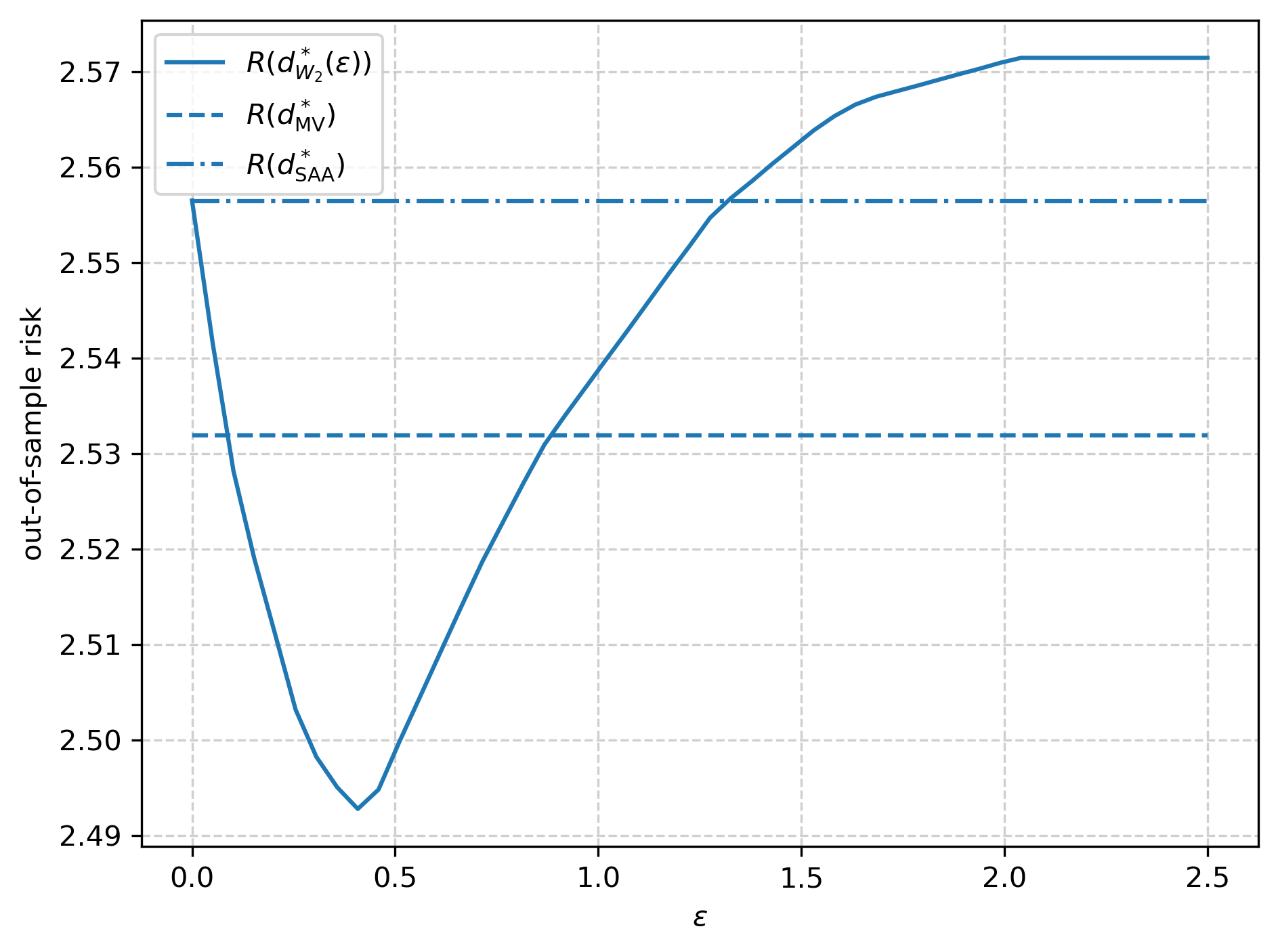}
  \caption{Out-of-sample performance of two robust models and SAA model}
  \label{fig:true-risk-misspec}
\end{figure}

\section*{Funding details}

This work was supported by 
the National Natural Science Foundation of China under Grant 12371476. 

\section*{Disclosure statement}
Declarations of interest: none

%
%
\appendix
\section*{Appendix}\label{app:proofs}

\subsection*{EC.1 Proof of Section \ref{main-sect}}
\subsection*{EC.1.1 Proof of Section \ref{main-mean-variance}}
To show Theorem \ref{thm-rho_p}, we need Sion's minimax theorem (see \cite{S58}). 

\begin{lemma}\label{le1}
(Sion's Minimax Theorem) Let $M$ be a compact convex set, $N$ be a convex set and $f$ be a real-valued function on $M \times N$ such that $f(x, \cdot)$ is upper semi-continuous and quasi-concave on $N$ for each $x \in M$ and $f(\cdot, y)$ is lower semi-continuous and quasi-convex on $M$ for each $y \in N$.\footnote{For a real-valued function $f$ on a convex set $M$, we say $f$ is quasi-convex on $M$ if $f(\lambda x+(1-\lambda) y) \leq \max \{f(x), f(y)\}$ for any $ x,\, y \in M$ and $\lambda \in[0,1]$; and $f$ is quasi-concave on a convex set $M$ if $f(\lambda x+(1-\lambda) y) \ge \min \{f(x), f(y)\}$ for any $ x,\, y \in M$ and $\lambda \in[0,1]$.} Then we have
$$
\inf _{x \in M} \sup _{y \in N} f(x, y)=\sup _{y \in N} \inf _{x \in M} f(x, y).
$$
\end{lemma}

\begin{proof}[Proof of Theorem \ref{thm-rho_p}]
By the definition of $\rho$, we can rewrite \eqref{eq:main-inner} as
\begin{align}\label{eq:rho_p-1}
     &~\sup_{\ell \in \Gamma} \sup_{F \in S(\mu, \sigma)}  \inf_{t \in \R} \left\{ t + \E^F\left[\ell\left(X\wedge d-t\right)\right]+(1+\theta)\E^F\left[(X-d)_+\right]\right\} \nonumber \\
     &:= \sup_{\ell \in \Gamma} \sup_{F \in S(\mu, \sigma)}  \inf_{t \in \R} \E^F\left[\tilde{\ell}(X,d,t)\right].
\end{align}
Denote  $g_{\ell}(t,F):= \E^F[\tilde{\ell}(X,d,t)]$. One can verify that $g_{\ell}(t,F)$ is linear in $F$ and convex in $t$.
Fix $\ell\in\Gamma$ and $F\in S(\mu, \sigma)$. Define $H_\ell(t,F):=t +\E^F\left[\ell\left(X\wedge d-t\right)\right]$ for $t\in\R$. Since $\ell$ is convex and satisfies $1\in\mbox{int}\partial \ell(\R)$, for each $\ell\in\Gamma$ there exist $y_1(\ell),\,y_2(\ell)\in\R$ with $y_1(\ell)\le y_2(\ell)$ such that $\ell'(y)<1$ for $y<y_1(\ell)$ and $\ell'(y)>1$ for $y>y_2(\ell)$.
Note that $X\wedge d$ takes values in $[0,d]$. Consequently, for each $\ell\in\Gamma$ and $F \in S(\mu, \sigma)$,
$$
   \frac{\partial H_\ell(t,F)}{\partial t}<0,~t \in (-\infty,-y_2(\ell))~~{\rm and}~~\frac{\partial H_\ell(t,F)}{\partial t}>0,~t \in (d-y_1(\ell),\infty),
$$
where ${\partial H_\ell(t,F)}/{\partial t}$ represents the right derivative of $H_\ell(t,F)$ with respect to $t$.
Therefore, all the minimizers of the inner problem of Equation \eqref{eq:rho_p-1} must lie within the  interval $[-y_2(\ell),\,d-y_1(\ell)]=:[t_{\ell,1},\,t_{\ell,2}]$. 
It is worth noting that $t_{\ell,1}$ and $t_{\ell,2}$ are independent of $F$.
Hence, we can rewrite problem \eqref{eq:rho_p-1} as
$$
   \sup_{\ell\in \Gamma}  \sup_{F \in S(\mu, \sigma)}\min_{t \in [t_{\ell,1},t_{\ell,2}]} g_{\ell}(t,F).
$$
By Lemma \ref{le1}, for any $\ell \in \Gamma$, we have 
\begin{align*}
    \sup_{F \in S(\mu, \sigma)}  \inf_{t\in\R} g_{\ell}(t,F)&=\sup_{F \in S(\mu, \sigma)}\min_{t \in [t_{\ell,1},t_{\ell,2}]} g_{\ell}(t,F)\\
    &= \min_{t \in [t_{\ell,1},t_{\ell,2}]} \sup_{F \in S(\mu, \sigma)} g_{\ell}(t,F) \ge \inf_{t \in \R} \sup_{F \in S(\mu, \sigma)} g_{\ell}(t,F),
\end{align*}
The reverse direction  of the above inequality is trivial. Hence, 
 \begin{align*}
        \sup_{F \in S(\mu, \sigma)}\inf_{t \in \R} g_{\ell}(t,F)
        &=   \inf_{t \in \R} \sup_{F \in S(\mu, \sigma)} g_{\ell}(t,F).
\end{align*}
Therefore, problem \eqref{eq:rho_p-1} is equivalent to
\begin{align}\label{eq:rho_p-3}
        \sup_{\ell \in \Gamma} \inf_{t \in \R} \sup_{F \in S(\mu, \sigma)}  g_{\ell}(t,F)=\sup_{\ell \in \Gamma} \inf_{t \in \R} \sup_{F \in S(\mu, \sigma)} \E^F[\tilde{\ell}(X,d,t)].
\end{align}
Note that $\E^F[\tilde{\ell}(X,d,t)]$ is an expected loss. By Proposition 6.39 in \cite{SDR21}, the worst-case distribution of Problem \eqref{eq:rho_p-3} can be taken to be a three-point distribution set $S_3(\mu, \sigma)$, that is,
\begin{align*}
     \sup_{F \in S(\mu, \sigma)} \E^F[\tilde{\ell}(X,d,t)] = \sup_{F\in S_3(\mu, \sigma)} \E^F[\tilde{\ell}(X,d,t)].
\end{align*}
Substituting this identity into \eqref{eq:rho_p-3} yields \eqref{eq:rho_p}. 
\end{proof}
\begin{proof}[Proof of Proposition \ref{pro1}]
    It suffices to show that
\begin{align}\label{ieq:pro1}
    \sup_{F \in S(\mu, \sigma)} \rho^F(X \wedge d)+(1+\theta)\E^F\left[(X-d)_+\right] \ge \sup_{F \in \overline{S}(\mu, \sigma)} \rho^F(X \wedge d)+(1+\theta)\E^F\left[(X-d)_+\right].
\end{align}
To this end, take $F \in \overline{S}(\mu, \sigma)$ and $X\sim F$.  Without loss of generality, assume ${\rm Var}^F(X)<\sigma^2$. To prove \eqref{ieq:pro1}, we consider the following two cases.

\emph{Case 1}. \ Suppose $\p(X>d)=0$. We find that there exists $x_0\in(0,d]$, such that $\p(x_0< X\le d)>0$; moreover, $\rho^F(X \wedge d)+(1+\theta)\E^F[(X-d)_+]=\rho^F(X)$. Let $U$ be a uniform random variable on $(0,1)$ such that $F^{-1}(U)=X$ almost surely (a.s.). The existence of such uniform random variable $U$ for any $X$ is guaranteed, for instance, by Lemma A.32 of \cite{FS16}. For $\epsilon>0$ and $\delta\in (0,1-F(x_0)]$, define
     $$
        X_{\epsilon,\delta} = \begin{cases}  0, & \textrm{if~} X \le x_0,\\
                 -{\epsilon},  & \textrm{if~} X> x_0, ~U\le 1-\delta,\\
                     p\epsilon/\delta,  & \textrm{if~} U>1-\delta, \end{cases}
     $$
     where $p=\p(X> x_0, U\le 1-\delta)=1-\delta-F(x_0)$. Taking $X+X_{\epsilon,\delta}\sim F_{\epsilon,\delta}$, we have $X+X_{\epsilon,\delta}\ge X-\epsilon$ and $\E^{F_{\epsilon,\delta}} [X+X_{\epsilon,\delta}]=\mu$. When $\delta_0=1-F(x_0)$, we have $X_{\epsilon,\delta_0}=X$ and, thus, $\var^{F_{\epsilon,\delta_0}}(X+X_{\epsilon,\delta_0}) =\var^F(X) <\sigma^2$. For any $\epsilon>0$,  we observe that $\lim_{\delta \to 0} \var^{F_{\epsilon,\delta}} (X+X_{\epsilon,\delta})=+\infty$. By the continuity of $\var^{F_{\epsilon,\delta}}(X+X_{\epsilon,\delta})$ with respect to $\delta$, there exists $\delta_\epsilon >0$ such that
     $$
              X_\epsilon:=X+X_{\epsilon,\delta_\epsilon}\ge X-\epsilon~~{\rm and}~ ~\var^{F_{\epsilon}}(X_\epsilon)=\sigma^2,
     $$
     where $ F_{\epsilon}$ is the cdf of $X_\epsilon$. As $\rho^{F_\epsilon}(X_\epsilon \wedge d)\ge \rho^{F}((X-\epsilon) \wedge d)=\rho^{F}(X)-\epsilon$, we find that for any sufficiently small $\epsilon>0$, there exists $X_\epsilon\sim F_\epsilon$ with $F_\epsilon\in S(\mu, \sigma)$ such that
     $$
         \rho^{F_\epsilon}(X_\epsilon \wedge d)+(1+\theta)\E^{F_\epsilon}[(X_\epsilon-d)_+] \ge \rho^{F}(X)-\epsilon.
     $$
     Hence, for $X\sim F \in \overline{S}(\mu, \sigma)$ with $\p(X>d)=0$, we have
     $$
          \lim_{\epsilon\to 0}\rho^{F_\epsilon}(X_\epsilon \wedge d)+(1+\theta)\E^{F_\epsilon}[(X_\epsilon-d)_+] \ge \rho^{F}(X),
     $$
     where $F_{\epsilon}\in S(\mu, \sigma)$.

\emph{Case 2}.\ Suppose $q:=\p(X>d)>0$. Then, for $\xi\in[0,q)$, define
     $$
          \widetilde{X}_{\xi} = \begin{cases}  d,  &  \mbox{ with probability }\xi/q,\\
                 \frac{x_2q-d\xi}{q-\xi},  &  \mbox{ with probability } 1-\xi/q, \end{cases}
     $$
     where $x_2=\E[X|X> d]$. Let $\widetilde{F}_\xi$ be the cdf of $\widetilde{X}_{\xi}$, and $\widehat{F}$ be the cdf of $[X|X\le d]$. Now, we consider a random variable ${X}_{\xi}$ with the cdf ${F_\xi}:=(1-q)\widehat{F}+q\widetilde{F}_\xi$. Note that $\E^{F_\xi}[{X}_{\xi}]=\mu$, $\var^{F_0}[{X}_{0}]\le \var^F [X]<\sigma^2$ and $\lim_{\xi \to q}\var^{F_\xi}({X}_{\xi})=\infty$. By the continuity of $\var^{F_{\xi}}({X}_{\xi})$ with respect to $\xi$, there exists $\xi^\ast\in(0,q)$ such that $\var^{F_{\xi^\ast}}(X_{\xi^\ast})=\sigma^2$. Moreover, one can verify that $X_{\xi^\ast} \wedge d=X \wedge d$ and $\E^{F_{\xi^\ast}}[(X_{\xi^\ast}-d)_+]=\E^{F}[(X-d)_+]$. Hence, for $X\sim F\in \overline{S}(\mu, \sigma)$ with $\p(X>d)>0$, we have
     $$
         \rho^{F_{\xi^\ast}}(X_{\xi^\ast} \wedge d)+(1+\theta)\E^{F_{\xi^\ast}}[(X_{\xi^\ast}-d)_+]
                =\rho^{F}(X \wedge d)+(1+\theta)\E^{F}[(X-d)_+],
     $$
     where $F_{\xi^\ast}\in S(\mu, \sigma)$. 

Combining the above two cases, the desired \eqref{ieq:pro1} follows. This completes the proof.
\end{proof}
\begin{proof}[Proof of Proposition \ref{pro: convexset}]
    By Theorem \ref{thm-rho_p} and Proposition \ref{pro1}, the problem \eqref{eq:main} under the ROCE risk measure is equivalent to 
    \begin{align} \label{eq:rOCE}
   \min_{d\ge 0} \max_{\gamma\in\Gamma}\min_{t \in \R}\, &\max_{x_i\in\R_+,\,p_i\in\R_+}   \sum_{i=1}^3 p_i \tilde{\ell}(x_i,d,t),\\
     &~~~~~~~\text {\rm  s.t. } 
\sum_{i=1}^3 x_ip_i=\mu, \ \ \sum_{i=1}^3 p_i=1,\ \ \sum_{i=1}^3 x_i^2p_i\le\mu^2+\sigma^2, \notag
\end{align}
which is a tractable problem. Next, we focus on the inner maximization problem of \eqref{eq:rOCE}. 
For the inner maximization problem of \eqref{eq:rOCE},  
by taking $x_ip_i = y_i$ for $i=1,2, 3$, we can rewrite it as 
\begin{align*} 
\max_{y_i\in\R_+,\,p_i\in\R_+}   & \sum_{i=1}^3 p_i\tilde{\ell}\left(\frac{y_i}{p_i},d,t\right),\\
 \text {\rm s.t. }~~~   &
\sum_{i=1}^3 y_i=\mu, \ \ \sum_{i=1}^3 p_i=1,\ \ \sum_{i=1}^3 p_i\left(\frac{y_i}{p_i}\right)^2\le\mu^2+\sigma^2. 
\end{align*}
This completes the  proof.
\end{proof}
\subsection*{EC.1.2 Proof of Section \ref{main:sec_sample}}
\begin{proof}[Proof of Proposition \ref{exp+cvar}]
	By Theorem \ref{thm-rho_p}, problem \eqref{eq:main} with $\rho (X)=\eta_1\E[X]+(1-\eta_1)\CVaR_{\xi}(X)$ reduces to
	\begin{align}\label{eq:co3}
		\min_{d\ge 0}\inf_{t\in \R} \sup_{F \in S_3(\mu, \sigma)} \E^F[\tilde{\ell}(X,d,t)],
	\end{align}
	where $\tilde{\ell}(x,d,t):=t+ \eta_2(x\wedge d-t)_+-\eta_1 (t-x\wedge d)_+ +(1+\theta)(x-d)_+$.
	\begin{itemize}
		\item[(1)] If $\eta_2<(1+\theta)$, then $\tilde{\ell}(x,d,t)$ is decreasing in $d\in \R_+$, and thus the optimal deductible $d^*=\infty$.
		
		\item[(2)]  If $\eta_2\ge (1+\theta)$, using the proof of Theorem \ref{thm-rho_p}, we can further reduce problem \eqref{eq:main} to 
    \begin{align*}
        \min_{0 \le t \le d}  \sup_{F \in S_3(\mu, \sigma)}  \E^F[\tilde{\ell}(X,d,t)].
    \end{align*}
    Fix $d\ge 0$. For $t \in [0,d]$, one can verify that $\tilde{\ell}(x,d,t)$ is concave in $x\in [t,\infty)$. Applying Jensen's inequality, for any three-point distribution, we can find a two-point distribution such that the objective function $\E[\tilde{\ell}(X,d,t)]$ of the two-point distribution is larger than that of the original. Specifically, for any $F=[x_1,p_1;x_2,p_2;x_3,p_3]\in S_3(\mu, \sigma)$, without loss of generality, assume $x_1\le t \le x_2\le d\le x_3$ and define $F_\delta= [x_1-\delta (1-p_1), p_1; \overline{x}+ \delta p_1, 1-p_1 ]$, where $\overline{x} = (p_2 x_2+p_3x_3)/(p_2+p_3)$ and $\delta\ge 0$. One can verify that $\E^{F_\delta}[X]=\mu$, ${\rm Var}^{F_0}(X)\le \sigma^2$ and ${\rm Var}^{F_\delta}(X)$ is increasing and continuous in $\delta\ge 0$, and thus, there exists $\delta\ge 0$ such that ${\rm Var}^{F_\delta}(X)=\sigma^2$. Note that $\E^F[\tilde{\ell}(X,d,t)] \le \E^{F_\delta}[\tilde{\ell}(X,d,t)]$, and $\sup_{F \in S_2(\mu, \sigma)}\E^F[\tilde{\ell}(X,d,t)]$ is increasing in $\sigma$, where $S_{2}(\mu, \sigma):=\{F \in S(\mu, \sigma): F \mbox{ is a two-point cdf}\}$. Therefore, the optimal distribution must be a two-point distribution in $S(\mu,\sigma)$. 
        Standard manipulation yields the optimal deductible  $d^*=0$ if $\theta^*\le \sigma^2/\mu^2$ and $\mu-\sigma(1-\theta^*)/(2\sqrt{\theta^*})$ otherwise, where $\theta^*=(1+\theta-\eta_1)/(1-\eta_1)$.
	\end{itemize}
	Combining the above two cases,  we complete the proof.
\end{proof}
\begin{proof}[Proof of Proposition \ref{pro_e}]
    Let $\rho=e_\beta$ for $\beta \in [1/2,1]$ and set $\nu=\beta/(1-\beta)$. Then the problem \eqref{eq:main} can be reformulated as problem \eqref{eq:rOCE} with $\tilde{\ell}_\gamma(x,d,t)=t+ \gamma\nu(x\wedge d-t)_+-\gamma (t-x\wedge d)_+ +(1+\theta)(x-d)_+$, $\Gamma=[1/\nu,1]$, and $t\in[0,d]$. Define the inner maximization problem of \eqref{eq:rOCE} by
    \begin{align*}
        g(d,\gamma,t)=\left\{\begin{array}{lll}&\max_{x_i\in\R_+,\,p_i\in\R_+}   \sum_{i=1}^3 p_i \tilde{\ell}_\gamma(x_i,d,t) ,\\
     &~~~\text {\rm s.t. } 
\sum_{i=1}^3 x_ip_i=\mu, \ \ \sum_{i=1}^3 p_i=1,\ \ \sum_{i=1}^3 x_i^2p_i\le\mu^2+\sigma^2.\end{array}\right.
    \end{align*}
   We further simplify $g(d,\gamma,t)$ in the case $\gamma \nu>1+\theta$. Fix $d\ge0$,  $\tilde{\ell}_\gamma(x, d, t)$ is concave in $x \in[t,\infty)$. Similar to the proof of Proposition \ref{exp+cvar}, applying Jensen's inequality, for any three-point distribution, we can find a two-point distribution such that the objective function $\E[\tilde{\ell}_\gamma(X,d,t)]$ of the two-point distribution is larger than that of the original. Thus, the inner maximization problem of \eqref{eq:rOCE}, $g(d,\gamma, t)$, can be attained in a two-point distribution. Consequently, we also have  $$g(d,\gamma,t) = \max_{p\in [\sigma^2/(\mu^2+\sigma^2),1]}  p \tilde{\ell}_\gamma \left(x_1,\,d,\,t\right) + (1- p ) \tilde{\ell}_\gamma \left(x_2,\,d,\,t\right),$$
    where $x_1=\mu-\sigma\sqrt{\frac{1-p}{p}}$,  $x_2=\mu+\sigma\sqrt{\frac{p}{1-p}}$, and $0\le x_1\le x_2$. 
Thus, we obtain the tractable reformulation for problem \eqref{eq:main} with $\rho=e_\beta$, $\beta\in [1/2,1]$. In particular, if $\nu\le 1+\theta$, then $\gamma\nu\le 1+\theta$ for $\gamma\in[1/\nu,1]$. Hence, $\tilde{\ell}_\gamma(x,d,t)$ is decreasing in $d\ge 0$ for every $\gamma\in[1/\nu,1]$, and therefore the optimal deductible $d^*=\infty$. 
\end{proof}
\subsection*{EC.2 Proof of Section \ref{main-wasser}}
\subsection*{EC.2.1 Proof of Section \ref{main-wasser-1}}
\begin{proof}[Proof of Proposition \ref{switch_wasser}]
    Substituting the definition of ROCE risk measure into problem \eqref{eq:main-wasser-inner}, we obtain that $$\sup_{F \in \mathcal{B}_{\varepsilon}(F_0)} \rho^F\left(X \wedge d\right)+(1+\theta)\E^F\left[\left(X-d\right)_+\right] =\sup_{F \in \mathcal{B}_{\varepsilon}(F_0)}\sup_{\ell \in \Gamma} \inf_{t \in \R} \E^F\left[\tilde{\ell}\left(X,d,t\right)\right].$$ Then, similar to the proof of Theorem \ref{thm-rho_p}, we have 
    \begin{align*}
     \sup_{F \in \mathcal{B}_{\varepsilon}(F_0)}\sup_{\ell \in \Gamma}  \inf_{t \in \R} \E^F\left[\tilde{\ell}(X,d,t)\right]=&\sup_{\ell \in \Gamma} \sup_{F \in \mathcal{B}_{\varepsilon}(F_0)} \inf_{t \in \R} \E^F\left[\tilde{\ell}(X,d,t)\right]\\
     =&\sup_{\ell \in \Gamma} \sup_{F \in \mathcal{B}_{\varepsilon}(F_0)} \inf_{t \in [t_{\ell,1},t_{\ell,2}]} \E^F\left[\tilde{\ell}(X,d,t)\right].
\end{align*}
Since the Wasserstein ball is a convex set, $[t_{\ell,1},t_{\ell,2}]$ is an interval independent of $F$, and the objective function is linear in distribution and convex in $t$, by Lemma \ref{le1}, it holds that for $\ell\in\Gamma$, 
\begin{align*}
    \sup_{F \in \mathcal{B}_{\varepsilon}(F_0)} \inf_{t \in \R} \E^F\left[\tilde{\ell}(X,d,t)\right]=&\sup_{F \in \mathcal{B}_{\varepsilon}(F_0)} \inf_{t \in [t_{\ell,1},t_{\ell,2}]} \E^F\left[\tilde{\ell}(X,d,t)\right]\\=&\inf_{t \in [t_{\ell,1},t_{\ell,2}]}\sup_{F \in \mathcal{B}_{\varepsilon}(F_0)}  \E^F\left[\tilde{\ell}(X,d,t)\right]\\
     \ge& \inf_{t \in \mathbb{R}}\sup_{F \in \mathcal{B}_{\varepsilon}(F_0)} \E^F\left[\tilde{\ell}(X,d,t)\right].
\end{align*}
The reverse direction of the above inequality is trivial. Thus, we complete the proof.
\end{proof}
\begin{proof}[Proof of Theorem \ref{pro_wasser_ball_general}]
    By Proposition \ref{switch_wasser}, we have 
    \begin{align*}
     \sup_{F \in \mathcal{B}_{\varepsilon}(F_0)} \rho^F\left(X \wedge d\right)+(1+\theta)\E^F\left[\left(X-d\right)_+\right] 
     =  \sup_{\ell \in \Gamma} \inf_{t \in \R}\, \sup_{F \in \mathcal{B}_{\varepsilon}(F_0)} \E^F\left[\tilde{\ell}\left(X,d,t\right)\right],
\end{align*}
where $\tilde{\ell}(x,d,t)$ is defined as \eqref{c_ga_func}. Recall that for $\ell\in\Gamma$, $\ell(y)=\max_{k\le K}a_{\ell,k}y+b_{\ell,k}$, where $0\le a_{\ell,1}<\ldots<a_{\ell,K}$, and that there exist $-\infty=h_0<h_1<\ldots<h_{K}<h_{K+1}=\infty$ such that  $\ell(y)=a_{\ell,k}y+b_{\ell,k}$ for $y\in (h_{k-1},h_k]$. Fix $d\ge 0$. If $t\in[d-h_k,\,d-h_{k-1})$ (equivalently, $d-t\in(h_{k-1},h_k]$), then $\tilde{\ell}(x,d,t)$ can be rewritten in the following two cases:
(i) $a_{\ell,k}\le 1+\theta$ and (ii) $a_{\ell,k}> 1+\theta$.

\begin{itemize}
    \item[(i)] If $a_{\ell,k}\le 1+\theta$, $\tilde{\ell}(x,d,t)=\max_{j\le k+1}\tilde{\ell}_j(x,d,t)$, where $\tilde{\ell}_j(x,d,t)=t+a_{\ell,j}(x-t)+b_{\ell,j}$ for $j=1,\ldots,k$ and $\tilde{\ell}_{k+1}(x,d,t)=t+a_{\ell,k}(d-t)+b_{\ell,k}+(1+\theta)(x-d)$.
    \item[(ii)] If $a_{\ell,k}> 1+\theta$, $\tilde{\ell}(x,d,t)=\max_{j\le k}\tilde{\ell}_j(x,d,t)$, where $\tilde{\ell}_j(x,d,t)=t+a_{\ell,j}(x\wedge d-t)+b_{\ell,j}$ for $j=1,\ldots,k-1$ and $\tilde{\ell}_k(x,d,t)=t+a_{\ell,k}(x\wedge d-t)+b_{\ell,k}+(1+\theta)(x-d)_+$.
\end{itemize}
Thus, we find that $\tilde{\ell}(x,d,t)=\max_{j}\tilde{\ell}_j(x,d,t)$ and that $\tilde{\ell}_j(x,d,t)$ is concave in $x$. Applying Theorem 8 in \cite{KENS19}, we have 
\begin{align}
     &\sup_{\ell \in \Gamma} \inf_{t \in \R}\, \sup_{F \in \mathcal{B}_{\varepsilon}(F_0)} \E^F\left[\tilde{\ell}(X,d,t)\right]\notag\\
     =&
\left\{\begin{array}{ll}\sup_{\ell\in\Gamma}\inf _{t\in\mathbb{R},\lambda\ge 0, s_i, z_{i, j}, v_{i, j}}  \lambda \varepsilon^p+\frac{1}{N} \sum_{i=1}^N s_i & \\
~~~~~~~~~\text {\rm s.t. }  {\left[-\tilde{\ell}_j(\cdot,d,t)\right]^*\left(z_{i, j}-v_{i ,j}\right)+\sigma_{\Xi}\left(v_{i, j}\right)-z_{i ,j} \widehat{x}_i+\phi(q)\lambda \left|\frac{z_{i,j}}{\lambda}\right|^q \leq s_i} &\forall i \leq N,~\forall j.\end{array}\right.
\label{worst-expectation}
\end{align}
Here $[-\tilde{\ell}_j(\cdot,d,t)]^*$ denotes the conjugate function of $-\tilde{\ell}_j(\cdot,d,t)$, defined by $[-\tilde{\ell}_j(\cdot,d,t)]^*(y)=\sup_{x\in\R}\{xy+\tilde{\ell}_j(x,d,t)\}$. Moreover, $\sigma_{\Xi}(\cdot)$ is the support function of set $\Xi$, defined by $\sigma_{\Xi}(v)=\sup_{\xi\in\Xi}v^\top \xi$, and $\phi(q)=(q-1)^{(q-1)}/q^q$ for $q>1$, with $\phi(1)=1$. In particular, for $\Xi=\mathbb{R}_+$, we have $\sigma_{\Xi}(v)=\iota_{\mathbb{R}_-}(v)$, where $\iota_{A}(a)=0$ if $a\in A$ and $\iota_{A}(a)=\infty$ otherwise. 
By the definition of convex conjugate function, for given $d$, if $t\in[d-h_k,d-h_{k-1})$, that is, $d-t\in (h_{k-1},h_{k}]$, then the explicit form of convex conjugate function of $-\tilde{\ell}_j(x,d,t)$ with respect to $x$ is given by the following.
\begin{itemize}
    \item[(i)] If $a_{\ell,k}\le 1+\theta$, we have $[-\tilde{\ell}_j(\cdot,d,t)]^*(y)=\left(1-a_{\ell, j}\right) t+b_{\ell, j}+\iota_{\left\{-a_{\ell, j}\right\}}(y)$ for $j=1,\ldots,k$ and $[-\tilde{\ell}_{k+1}(\cdot,d,t)]^*(y)=(1-a_{\ell,k})t+(a_{\ell,k}-(1+\theta))d+b_{\ell,k}+\iota_{\{-(1+\theta)\}}(y)$.
    \item[(ii)] If $a_{\ell,k}> 1+\theta$, we have $[-\tilde{\ell}_j(\cdot,d,t)]^*(y)=\left(1-a_{\ell, j}\right) t+b_{\ell, j}+\left(y+a_{\ell, j}\right)d+\iota_{[-a_{\ell,j},  0]}(y)$ for $j=1,\ldots,k-1$ and $[-\tilde{\ell}_k(\cdot,d,t)]^*(y)=\left(1-a_{\ell, k}\right) t+b_{\ell, k}+\left(y+a_{\ell, k}\right)d+\iota_{[-a_{\ell,k},-(1+\theta)]}(y)$.
\end{itemize}
Substituting to \eqref{worst-expectation}, we complete the proof. 
\end{proof}
\subsection*{EC.2.2 Proof of Section \ref{wasser_sample}}
\begin{proof}[Proof of Proposition \ref{wasser_CVaR}]
	By Proposition \ref{switch_wasser}, problem \eqref{eq:main_wasser} with $\rho (X)=\eta_1\E[X]+(1-\eta_1)\CVaR_{\xi}(X)$, reduces to
	\begin{align*}
		\min_{0\le t\le d}\sup_{F \in \mathcal{B}_{\varepsilon}(F_0)} \E^F[\tilde{\ell}(X,d,t)],
	\end{align*}
	where $\tilde{\ell}(x,d,t) = t+\eta_2(x\wedge d-t)_+-\eta_1(t-x\wedge d)_++(1+\theta)(x-d)_+.$
	\begin{itemize}
		\item[(1)] If $\eta_2\le 1+\theta$, then $\tilde{\ell}(x,d,t)$ is decreasing in $d\in \R_+$, and thus, the optimal deductible $d^*=\infty$.
		
		\item[(2)]  If $\eta_2> 1+\theta$, we have $\tilde{\ell}(x,d,t)=\max\{\tilde{\ell}_1(x,d,t),\tilde{\ell}_2(x,d,t)\}$, where $\tilde{\ell}_1(x,d,t)=t+\eta_1(x-t)$ and $\tilde{\ell}_2(x,d,t)=t+\eta_2(x\wedge d-t)+(1+\theta)(x-d)_+$.  Recall that $\iota_{A}(\cdot)$ is the indicator function of the set $A$, defined by $\iota_{A}(a)=0$ if $a\in A$ and $\iota_{A}(a)=\infty$ otherwise. One can verify that $\tilde{\ell}_1(x,d,t)$ and $\tilde{\ell}_2(x,d,t)$ are concave in $x$, $[-\tilde{\ell}_{1}(\cdot,d,t)]^*(y)=\left(1-\eta_1\right) t+\iota_{\left\{-\eta_1\right\}}(y)$ and 
$[-\tilde{\ell}_2(\cdot,d,t)]^*(y)=\left(1-\eta_2\right) t+\left(y+\eta_2\right)d+\iota_{[-\eta_2,-(1+\theta)]}(y) $. Applying Theorem 8 in \cite{KENS19}, we obtain the finite-dimensional reformulation directly. 
	\end{itemize}
	Combining the above two cases,  we complete the proof.
\end{proof}
\begin{proof}[Proof of Proposition \ref{expectile_wasser}]
	By Proposition \ref{switch_wasser}, problem \eqref{eq:main_wasser} with $\rho =e_\beta$ reduces to
	\begin{align*}
		\min_{d\ge 0}\sup_{\gamma\in[1/\nu,1]}\inf_{t\in [0,d]}\sup_{F \in \mathcal{B}_{\varepsilon}(F_0)} \E^F[\tilde{\ell}_\gamma(X,d,t)],
	\end{align*}
	where $\tilde{\ell}_\gamma(x,d,t) = t+\gamma\nu(x\wedge d-t)_+-\gamma(t-x\wedge d)_++(1+\theta)(x-d)_+.$
	\begin{itemize}
		\item[(1)] If $\nu\le(1+\theta)$, then for $\gamma\in[1/\nu,1]$, $\tilde{\ell}_\gamma(x,d,t)$ is decreasing in $d\in \R_+$, and thus, the optimal deductible $d^*=\infty$.
		
		\item[(2)]  If $\nu> (1+\theta)$, we consider the following two cases.
        \begin{itemize}
            \item[(i)] For $\gamma\in[1/\nu,(1+\theta)/\nu]$, we have $\tilde{\ell}_\gamma(x,d,t)=\max\{\tilde{\ell}_{\gamma,1}(x,d,t),\tilde{\ell}_{\gamma,2}(x,d,t),\tilde{\ell}_{\gamma,3}(x,d,t)\}$, where $\tilde{\ell}_{\gamma,1}(x,d,t)=t+\gamma(x-t)$, $\tilde{\ell}_{\gamma,2}(x,d,t)=t+\gamma\nu(x-t)$ and $\tilde{\ell}_{\gamma,3}(x,d,t)=t+\gamma\nu(d-t)+(1+\theta)(x-d)$. Recall that $\iota_{A}(\cdot)$ is the indicator function of the set $A$, defined by $\iota_{A}(a)=0$ if $a\in A$ and $\iota_{A}(a)=\infty$ otherwise. One can verify that $\tilde{\ell}_{\gamma,1}(x,d,t)$, $\tilde{\ell}_{\gamma,2}(x,d,t)$ and $\tilde{\ell}_{\gamma,3}(x,d,t)$ are concave in $x$, $[-\tilde{\ell}_{\gamma,1}(\cdot,d,t)]^*(y)=\left(1-\gamma\right) t+\iota_{\left\{-\gamma\right\}}(y)$, 
$[-\tilde{\ell}_{\gamma,2}(\cdot,d,t)]^*(y)=\left(1-\gamma\nu\right) t+\iota_{\{-\gamma\nu\}}(y)$ and $[-\tilde{\ell}_{\gamma,3}(\cdot,d,t)]^*(y)=\left(1-\gamma\nu\right) t+(\gamma\nu-(1+\theta))d+\iota_{\{-(1+\theta)\}}(y)$. Applying Theorem 8 in \cite{KENS19}, we obtain the corresponding finite-dimensional reformulation. 
            \item[(ii)] For $\gamma\in((1+\theta)/\nu,1]$, we have $\tilde{\ell}_{\gamma}(x,d,t)=\max\{\tilde{\ell}_{\gamma,1}(x,d,t),\tilde{\ell}_{\gamma,2}(x,d,t)\}$, where $\tilde{\ell}_{\gamma,1}(x,d,t)=t+\gamma(x-t)$ and $\tilde{\ell}_{\gamma,2}(x,d,t)=t+\gamma\nu(x\wedge d-t)+(1+\theta)(x-d)_+$. One can verify that $\tilde{\ell}_{\gamma,1}(x,d,t)$ and $\tilde{\ell}_{\gamma,2}(x,d,t)$ are concave in $x$, $[-\tilde{\ell}_{\gamma,1}(\cdot,d,t)]^*(y)=\left(1-\gamma\right) t+\iota_{\left\{-\gamma\right\}}(y)$ and 
$[-\tilde{\ell}_{\gamma,2}(\cdot,d,t)]^*(y)=\left(1-\gamma\nu\right) t+\left(y+\gamma\nu\right)d+\iota_{[-\gamma\nu,-(1+\theta)]}(y) $. Applying Theorem 8 in \cite{KENS19}, we obtain the corresponding finite-dimensional reformulation. 
        \end{itemize}
	\end{itemize}
	Combining the above two cases,  we complete the proof.
\end{proof}
\end{document}